\documentclass[a4paper,10pt]{amsart}

\usepackage[margin=1in]{geometry}
\usepackage{microtype}

\usepackage{amsmath,amssymb,amsthm,mathtools}

\usepackage{graphicx}
\usepackage{xcolor}
\usepackage{tikz}
\usetikzlibrary{positioning}
\usepackage{tikz-cd}
\usepackage{subcaption}
\usepackage{booktabs}
\usepackage{multicol}

\usepackage[
    backend=biber,
    style=numeric,
    sorting=nty,
    sortcites=true
]{biblatex}
\usepackage[
    colorlinks=true,
    linkcolor=blue,
    citecolor=red,
    urlcolor=blue
]{hyperref}

\theoremstyle{plain}
\newtheorem{theorem}{Theorem}[section]
\newtheorem{lemma}[theorem]{Lemma}
\newtheorem{proposition}[theorem]{Proposition}
\newtheorem{corollary}[theorem]{Corollary}
\newtheorem{conjecture}[theorem]{Conjecture}
\newtheorem{open}[theorem]{Open Problem}

\theoremstyle{definition}
\newtheorem{definition}[theorem]{Definition}
\newtheorem{remark}[theorem]{Remark}

\DeclareMathOperator{\Aut}{Aut}
\DeclareMathOperator{\Fix}{Fix}
\DeclareMathOperator{\End}{End}
\DeclareMathOperator{\REnd}{REnd}
\DeclareMathOperator{\RAut}{RAut}
\DeclareMathOperator{\Cong}{Cong}
\DeclareMathOperator{\Sym}{Sym}

\newcommand{\FixV}[1]{\Fix_{V}(#1)}

\newcommand{\id}{\mathrm{id}}

\title{On totally synchronizing graphs}



\author{Daniele D'Angeli$^{\ast}$} 
\address{$\ast$ Universit\`a Niccolo Cusano, Rome, Italy}
\email{daniele.dangeli@unicusano.it}

\author{Emanuele Rodaro$^{\S}$}
\address{$\S$ Politecnico di Milano, Milan, Italy}
\email{emanuele.rodaro@polimi.it}

\begin{document}

\begin{abstract}
A coloring of a finite $k$-out directed graph $G$ is viewed as a deterministic complete automaton with state set $V(G)$. The graph $G$ is called \emph{totally synchronizing} if every coloring is synchronizing.
We prove that total synchronization imposes strong restrictions on symmetry: if $G$ is strongly connected and totally synchronizing, then $\Aut(G)$ contains no semiregular element; in particular, if $|\Aut(G)|$ is divisible by a prime $p>k$, then $G$ is not totally synchronizing.
We then give general constructions of strongly connected $k$-out graphs with prescribed quotients and prescribed automorphism group that are \emph{not} totally synchronizing.
On the quotient side, we relate graph congruences to strong lumpability of the uniform random walk on $G$ and introduce \emph{totally simple} graphs, characterized by the absence of nontrivial congruences.
In this setting we obtain a Perron--Frobenius sufficient condition for total synchronization: a strongly connected non-lumpable graph whose integer Perron--Frobenius eigenvector admits at most one nontrivial equipartition is totally synchronizing.
Finally, we show that deciding whether a primitive $k$-out graph admits a non-synchronizing coloring is NP-complete, resolving an open problem of Gusev--Szyku{\l}a, and prove NP-completeness of deciding whether a graph admits a nontrivial Eulerian lumping.
\end{abstract}

\maketitle

\section{Introduction}
Let $G$ be a finite directed graph with vertex set $Q$ and edge set $E$, in which every vertex has out-degree $k$; such a graph is called a \emph{$k$-out graph}. A \emph{coloring} (or \emph{labeling}) of $G$ by an alphabet $\Sigma$ with $|\Sigma|=k$ is a map $\ell:E\to\Sigma$ such that, for every vertex $q\in Q$ and every letter $a\in\Sigma$, there is a unique edge leaving $q$ with label $a$. Equivalently, such a coloring defines a deterministic complete automaton $\mathcal A_\ell=(Q,\Sigma,\delta_\ell)$, $\delta_\ell(q,a)=q'
\ \text{iff}\ (q\xrightarrow{a}q')\in E.$
The automaton $\mathcal A_\ell$ is \emph{synchronizing} if there exists a word
$w\in\Sigma^*$ such that $|\delta_\ell(Q,w)|=1$; any such word is called a
\emph{reset} (or \emph{synchronizing}) word.
Synchronizing automata have been extensively studied, largely motivated by the
celebrated \v{C}ern\'y conjecture, which asserts that every synchronizing
automaton with $n$ states admits a reset word of length at most $(n-1)^2$
(see~\cite{Ce64}).
Over the years, this conjecture has stimulated a vast literature on
synchronizing automata, including algorithms for computing short reset words,
proofs of the conjecture for important classes of automata, and connections
with coding theory, symbolic dynamics, and the Road Coloring Problem; see, for
instance,
\cite{AlRo, BeBePe, Dubuc, Epp, Kari, Ro18Adv, SteinbEJC, Steinb, Trah,
trahtman2009, Vo_CIAA07, Don, BehJoh, Perrin_unamb_coded_shifts}.
For a recent survey of the classes of automata for which the \v{C}ern\'y
conjecture is known to hold, see~\cite{survey_volkov_2025}. Besides the \v{C}ern\'y conjecture, synchronization has also been extensively studied from the perspective of permutation groups through the theory of synchronizing groups. This line of research, developed by Ara\'ujo, Cameron, Steinberg and collaborators, has revealed deep connections between automata, permutation groups, graph endomorphisms and algebraic combinatorics; see, for instance,
\cite{AraujoCameronJCTB2014,AraujoEtAlPLMS2016,AraujoCameronSteinberg2017}.

The present paper approaches synchronization from a graph-theoretic viewpoint by regarding an automaton as a coloring of a $k$-out graph. This viewpoint is also natural from symbolic dynamics, since a coloring is precisely a complete right-resolving edge-labeling, and therefore determines a right-resolving presentation of a sofic shift; see, for instance, \cite{LM}. This perspective lies at the origin of the Road Coloring Problem of Adler, Goodwyn, and Weiss~\cite{adler1977}, which asked whether every \emph{primitive} $k$-out graph admits at least one synchronizing coloring. Trahtman’s theorem answers this affirmatively~\cite{trahtman2009}: primitivity (strong connectivity and aperiodicity) is exactly the right unlabeled condition guaranteeing the existence of a synchronizing labeling. See also Volkov’s survey~\cite{volkov2008}. In this paper we study a different aspect of the same setting. Rather than asking for the existence of \emph{some} synchronizing coloring, we ask to what extent synchronization is forced by the unlabeled directed skeleton itself. From the symbolic-dynamical viewpoint, each coloring $\ell$ yields a deterministic presentation $(G,\ell)$, equivalently a $1$--block factor of the edge shift of $G$, and varying $\ell$ amounts to varying the symbolic coding of the same underlying Markov shift. This suggests measuring how \emph{typical} synchronization is among all such codings. Let $\mathcal C(G)$ denote the set of all colorings of $G$. Following~\cite{GusevSzykula15}, we define the \emph{synchronizing ratio}
\[
\mathrm{SynRatio}(G)
:=
\frac{\bigl|\{\ell\in\mathcal C(G):\mathcal A_\ell\text{ is synchronizing}\}\bigr|}{|\mathcal C(G)|}.
\]
The main protagonists of this paper are the \emph{totally synchronizing} graphs introduced in~\cite{GusevSzykula15}: we say that $G$ is totally synchronizing if $\mathrm{SynRatio}(G)=1$, that is, if every coloring of $G$ is synchronizing. This is an extremal robustness property: synchronization is not a consequence of a special choice of labels, but an invariant of the underlying directed graph. The study of synchronizing ratios suggests that total synchronization should be far from exceptional. In~\cite{GusevSzykula15}, Gusev and Szyku{\l}a propose two complementary conjectures. The first predicts a uniform lower bound for primitive $k$-out graphs $\mathrm{SynRatio}(G)\ \ge\ \frac{k-1}{k},$ up to a single explicit small exception. The second predicts that for each fixed $k\ge 2$, a random primitive $k$-out graph on $n$ vertices is totally synchronizing with probability tending to $1$ as $n\to\infty$. In this sense, totally synchronizing graphs should represent the typical behaviour, while graphs supporting many non-synchronizing colorings should be rare. Somewhat paradoxically, despite this apparent abundance, a satisfactory structural characterization of totally synchronizing graphs seems out of reach at present.  Beyond sufficient criteria (for instance those based on Perron--Frobenius data in~\cite{GuPri}), there is currently no general description in purely graph-theoretic terms for the class of totally synchronizing graphs.

In this paper we continue the study of totally synchronizing $k$-out graphs from a graph-theoretic viewpoint.
We focus on two kinds of structure already present in the unlabeled graph: symmetries and quotients.
The symmetry part gives graph-level obstructions to total synchronization.
Our first main result shows that if $G$ is strongly connected and totally synchronizing, then $\Aut(G)$ contains no semiregular element (Theorem~\ref{theo: main}).
As an immediate consequence, if $G$ is strongly connected and $|\Aut(G)|$ is divisible by a prime $p>k$, then $G$ is not totally synchronizing (Corollary~\ref{cor:prime-divisor-greater-than-k}).
Our second main contribution goes in the opposite direction.
We give general constructions of strongly connected $k$-out graphs that are not totally synchronizing while allowing one to prescribe both a quotient graph and substantial symmetry.
More precisely, for any strongly connected base graph and any nontrivial finite group $H$, we construct a strongly connected refinement together with a congruence whose quotient is the base graph. The kernel of the induced action of $\Aut(G)$ on the vertex set consists precisely of the automorphisms that fix every vertex, and the corresponding quotient is isomorphic to $H$. In the parallel-free and loop-free case, one can arrange that $\Aut(G)\cong H$ (Corollary~\ref{cor:one-loop-per-vertex-parallel-free}).

The quotient part is expressed through lumpability.
Every $k$-out graph carries a natural uniform random walk, and graph congruences are precisely the strong lumpings of this Markov chain.
We introduce \emph{totally simple} graphs, meaning graphs whose every coloring is a simple automaton, and prove that total simplicity is equivalent to the absence of nontrivial lumpable partitions (Theorem~\ref{theo: lump=to simple}).
This gives a structural setting in which quotient obstructions have been removed.
As an illustration, we prove that the underlying $2$-out graphs of the \v{C}ern\'y automata are totally simple (Proposition~\ref{prop:cerny-underlying-totally-simple}).
We also prove a slight extension of a Perron--Frobenius criterion of Friedman and Gusev--Pribavkina: if $G$ is strongly connected and non-lumpable, and its integer-normalized Perron--Frobenius eigenvector admits at most one nontrivial equipartition, then $G$ is totally synchronizing (Theorem~\ref{theo: t-s for unique partition}).

Finally, we address algorithmic questions suggested by the preceding structural results.
Our main decision result is that deciding whether a primitive $k$-out graph admits a non-synchronizing coloring is NP-complete (Theorem~\ref{thm:NTS-NPcomplete}), thereby resolving an open problem posed in~\cite{GusevSzykula15}.
We also prove that deciding whether a graph admits a nontrivial Eulerian lumping is NP-complete (Theorem~\ref{thm:eulerian-lumping-npcomplete}).
On the counting side, we establish a reduction from counting labeled $1$-factorizations of $k$-regular bipartite graphs to counting non-synchronizing colorings of $k$-out graphs (Proposition~\ref{prop:encode-1fact-into-nsc}).

\section{Preliminaries}\label{sec: preliminaries}
A \emph{directed multigraph}, or \emph{quiver}, is a tuple $G=(V,E,s,t)$ consisting of two finite sets $V$ and $E$
(sometimes denoted $V(G)$ and $E(G)$) and two maps $s,t:E\to V$.
The elements of $V$ are called \emph{vertices} and the elements of $E$ are called \emph{edges}.
For an edge $e\in E$, the vertex $s(e)$ is its \emph{source} and $t(e)$ its \emph{target}. For subsets $A,B\subseteq V$, we define
\[
\delta_{A}(B)\;:=\;\{\,e\in E:\ s(e)\in A \text{ and } t(e)\in B\,\}.
\]
When $A$ or $B$ is a singleton we omit braces, e.g.\ $\delta_A(v):=\delta_A(\{v\})$ and $\delta_u(B):=\delta_{\{u\}}(B)$. Throughout the paper, we refer to directed multigraphs/quivers simply as \emph{graphs}.
\begin{definition}[homomorphism, epimorphism, monomorphism]
    A \emph{homomorphism} $\varphi: G \to H$ of graphs $G = (V, E, s, t)$ and $H = (W, F, s, t)$ is a pair $\varphi = (\varphi_{v}, \varphi_{e})$ of mappings $\varphi_{v}: V \to W$ and $\varphi_{e}: E \to F$ commuting with $s,t$, that is, $s(\varphi_{e}(e)) = \varphi_{v}(s(e))$ and $t(\varphi_{e}(e)) = \varphi_{v}(t(e))$ hold for every edge $e \in E$. The homomorphism $\varphi$ is called an \emph{epimorphism} if both $\varphi_{v}$ and $\varphi_{e}$ are surjective. It is a \emph{monomorphism} if both $\varphi_{v}$ and $\varphi_{e}$ are injective. If it is both a monomorphism and epimorphism it is called an \emph{isomorphism}. 
    
    Given an equivalence relation $\rho$ on $V$, the quotient digraph $G/\rho$ is defined as follows: its vertex set is $V(G/\rho)=V(G)/\rho$, and for each pair of classes $[u]_\rho,[v]_\rho$ there are exactly
\[
\max_{p\in [u]_\rho}\bigl|\{ e \in E : s(e)=p,\ t(e)\in [v]_\rho \}\bigr|
\]
parallel edges from $[u]_\rho$ to $[v]_\rho$.
\end{definition}
 A graph $G=(V,E,s,t)$ is called \emph{$k$-out} if every vertex has exactly $k$ outgoing edges, that is, $|\delta_u(V)|=k$ for all $u\in V$. These graphs are exactly the underlying graphs of deterministic and complete (semi)automata. We recall that a (complete) deterministic \emph{automaton} is a triple $\mathcal A=(Q,\Sigma,\delta)$ where $Q$ is a finite
set of states, $\Sigma=[k]:=\{1,\ldots, k\}$ a finite alphabet, and $\delta:Q\times\Sigma\to Q$ the transition function.
When $\delta$ is clear we write $p\cdot a:=\delta(p,a)$ for $p\in Q$, $a\in\Sigma$, and extend to words
$u=a_1\cdots a_\ell\in\Sigma^*$ by $p\cdot u:=(((p\cdot a_1)\cdot a_2)\cdots)\cdot a_\ell$ (with
$p\cdot\varepsilon=p$). For $S\subseteq Q$ we set $S\cdot u:=\{p\cdot u: p\in S\}$.
\\
Let $G=(V,E,s,t)$ be a $k$-out graph.
An \emph{automaton coloring} of $G$ is a map $\chi:E\to[k]$ such that for every vertex $v\in V$ the restriction
$\chi\big|_{\delta_v(V)}$ is a bijection onto $[k]$.
Given such a coloring, we obtain an automaton
\[
  \chi(G):=(V,[k],\delta_\chi)
\]
by declaring $\delta_\chi(v,a)=t(e)$ where $e$ is the unique edge with $s(e)=v$ and $\chi(e)=a$.
We denote by $\mathcal C(G)$ the set of all automaton colorings of $G$.
An \emph{endomorphism} of the automaton $\mathcal{A}$ is a map $\varphi : Q \to Q$ such that
\[
\delta(\varphi(q), a) = \varphi(\delta(q,a))
\quad
\text{for all } q \in Q \text{ and } a \in \Sigma.
\]
The set of all endomorphisms of $\mathcal{A}$ is denoted by $\End(\mathcal{A})$.
An \emph{automorphism} of $\mathcal{A}$ is a bijective endomorphism, and the group of automorphisms is denoted by $\Aut(\mathcal{A}) \subseteq \End(\mathcal{A})$.

 Let $\mathcal A=(Q,\Sigma,\delta)$ be a deterministic complete automaton. A subset $S\subseteq Q$ is \emph{synchronizable} if there exists a word $u\in\Sigma^*$ such that $|S\cdot u|=1$. The automaton $\mathcal A$ is called \emph{synchronizing} if $Q$ is synchronizing, i.e.,  there exists a word $w\in\Sigma^*$ such that \(|Q\cdot w|=1.\) Any such word is called a \emph{reset word}. 
 
 An equivalence relation $\rho$ on $Q$ is an \emph{automaton congruence} if for all $p,q\in Q$ and all $a\in\Sigma$,
\[
p\mathrel{\rho}q \Longrightarrow p\cdot a \mathrel{\rho}q\cdot a.
\]
We denote by $\Cong(\mathcal A)$ the set of all automaton congruences of $\mathcal A$. The automaton $\mathcal A$ is called \emph{simple} if its only congruences are the discrete and the universal one.





\subsection{Endomorphisms and automorphisms of $k$-out graphs}
From now on, we will consider $k$-out graphs, which form the main object of study of the present paper. Recall that a graph $G=(V,E,s,t)$ is called \emph{$k$-out} if every vertex has exactly $k$ outgoing edges, that is, $|\delta_u(V)|=k$ for all $u\in V$.
If $G$ and $H$ are $k$-out graphs, a homomorphism of quivers $\varphi:G\to H$ is called a \emph{$k$-out homomorphism}, or equivalently a \emph{local out-isomorphism}, if for every $v\in V(G)$ the induced map
\[
\{e\in E(G): s(e)=v\}\longrightarrow \{f\in E(H): s(f)=\varphi_v(v)\},\qquad e\mapsto \varphi_e(e),
\]
is a bijection.
For a $k$-out graph $G$, we denote by $\End(G)$ the monoid of $k$-out endomorphisms of $G$.
The notion of $k$-out homomorphism defines a natural category whose objects are $k$-out graphs and whose morphisms are local out-isomorphisms. In this category, the isomorphisms are exactly the isomorphisms of the underlying quivers. Equivalently, a homomorphism
\[
\varphi=(\varphi_v,\varphi_e):G\to H
\]
between $k$-out graphs is an isomorphism if and only if it is a local out-isomorphism and both $\varphi_v$ and $\varphi_e$ are bijective.

In the present directed setting, $k$-out homomorphisms are the natural analogue of covering maps: they are locally bijective on outgoing stars. This is weaker than the usual topological notion of graph covering, which is locally bijective on the whole star of each vertex.

Accordingly, for a $k$-out graph $G$ we denote by $\Aut(G)$ both the automorphism group of the underlying graph and, equivalently, the group of invertible elements of $\End(G)$.

In particular, every $\varphi=(\varphi_v,\varphi_e)\in\Aut(G)$ preserves both outgoing and incoming multiplicities. More precisely, for every $x\in V(G)$ and every $A\subseteq V(G)$,
\[
|\delta_x(A)|=|\delta_{\varphi_v(x)}(\varphi_v(A))|
\qquad\text{and}\qquad
|\delta_A(x)|=|\delta_{\varphi_v(A)}(\varphi_v(x))|.
\]

In the category of directed multigraphs, parallel edges give rise to unavoidable endomorphisms: a map may fix every vertex and still permute edges with the same source and target. Since these symmetries are intrinsic to the category and disappear only in the parallel-free setting, we isolate them explicitly. For a $k$-out graph $G$, we write
\[
\End_E(G):=\{\varphi\in\End(G):\varphi_v=\id_{V(G)}\},
\qquad
\Aut_E(G):=\Aut(G)\cap \End_E(G).
\]
Thus $\End_E(G)$ consists of the endomorphisms fixing all vertices, and $\Aut_E(G)$ of the automorphisms fixing all vertices. We call them the monoid of \emph{vertex-trivial endomorphisms} and the group of \emph{vertex-trivial automorphisms} of $G$, respectively. It is normal in $\Aut(G)$ since conjugation preserves the induced action on vertices.

Note that the notion of $k$-out homomorphism is weaker than local injectivity in the usual edge-labelled sense, where one also requires preservation of the local ordering or labels of outgoing edges. Nevertheless, in the strongly connected case no nontrivial collapse can occur.

\begin{proposition}\label{prop:strongly-connected-endo-auto}
Let $G$ be a strongly connected $k$-out graph. Then every $k$-out endomorphism $\varphi:G\to G$ is an epimorphism and hence, by finiteness, an automorphism. In particular, $\End(G)=\Aut(G)$.
\end{proposition}

\begin{proof}
Let $G=(V,E,s,t)$ and let $\varphi=(\varphi_v,\varphi_e)\in\End(G)$. If $v\in \varphi_v(V)$ and $f\in E$ starts at $v$, then, since $\varphi$ is a $k$-out homomorphism, the outgoing star of $v$ is entirely in the image of $\varphi_e$. Hence $f=\varphi_e(e)$ for some edge $e\in E$. Assume $V\neq\varnothing$, and fix $v\in \varphi_v(V)$. Let $w\in V$ be arbitrary. Since $G$ is strongly connected, there exists a path
\[
v=v_0\xrightarrow{f_1}v_1\xrightarrow{f_2}\cdots\xrightarrow{f_\ell}v_\ell=w.
\]
Choose $u_0\in V$ such that $\varphi_v(u_0)=v_0$. As observed above, there exists an edge $e_1$ with $\varphi_e(e_1)=f_1$. Then
\[
v_1=t(f_1)=t(\varphi_e(e_1))=\varphi_v(t(e_1)),
\]
so $v_1\in \varphi_v(V)$. Repeating the same argument along the path shows that every $v_i$, and in particular $w$, lies in the image of $\varphi_v$. Thus $\varphi_v$ is surjective. Since $\varphi$ is locally bijective on outgoing stars, surjectivity of $\varphi_v$ implies surjectivity of $\varphi_e$ as well. Hence $\varphi$ is an epimorphism. Since $G$ is finite, $\varphi_v$ and $\varphi_e$ are bijective, so $\varphi\in\Aut(G)$.
\end{proof}

Let $G$ be a $k$-out graph and let $\varphi\in\End(G)$. For a vertex $v\in V(G)$, the set
\[
O_\varphi(v):=\{v,\varphi(v),\varphi^2(v),\dots\}
\]
is the \emph{forward orbit} of $v$. Since $\varphi$ need not be surjective, the graph defined by the iterates of $\varphi$ decomposes into components, each containing a unique cycle, called the \emph{limit cycle} of that component and denoted by $C_v$.

The corresponding \emph{basin} is
\[
B_\varphi(v):=\bigcup_{n\ge 0}\varphi^{-n}(C_v),
\]
that is, the set of all vertices whose forward orbit eventually enters the cycle $C_v$. The basins form a partition of $V(G)$.

\subsection{Lumpability and congruences of $k$-out graphs}

For $k$-out graphs, the natural notion of quotient is not an arbitrary partition of the vertex set, but a partition compatible with the outgoing multiplicities. This is precisely the notion of strong lumpability for the uniform random walk on the graph, and it is the correct analogue of equitable partitions in the present category.

Let $G=(V,E,s,t)$ be a finite $k$-out graph. It defines a Markov chain with transition matrix
\[
P(u,v)=\frac{|\delta_u(v)|}{k},
\qquad u,v\in V,
\]
that is, the random walk choosing uniformly among the outgoing edges. If $G$ is strongly connected, this chain is irreducible and hence admits a unique stationary distribution; see, e.g.,~\cite{KemenySnellFMC}. In this setting, strong lumpability admits a simple graph-theoretic reformulation.

\begin{definition}[Lumpability and congruence]
Let $G=(V,E,s,t)$ be a $k$-out graph, and let $\mathcal C=\{C_1,\dots,C_m\}$ be a partition of $V$. We say that $\mathcal C$ is \emph{lumpable} if for every pair of blocks $C_i,C_j$ and all $u,u'\in C_i$ one has
\[
|\delta_u(C_j)|=|\delta_{u'}(C_j)|.
\]
Equivalently, for each fixed $i$ and $j$, every vertex in $C_i$ has the same number of outgoing edges into $C_j$. 
\\
An equivalence relation $\rho$ on $V(G)$ is called a \emph{congruence} (of the $k$-out graph $G$) if the quotient graph $G/\rho$ is again $k$-out.
\end{definition}

With our definition of quotient, congruences and lumpable partitions are the same notion. Indeed, if $\rho$ is lumpable, then the quotient is well defined and $k$-out. Conversely, if $G/\rho$ is $k$-out and $C=[u]_\rho$ is a class, then for every target class $D$ one has
\[
|\delta_x(D)|\le \max_{p\in C} |\delta_p(D)|
\qquad(x\in C),
\]
and summing over all target classes gives equality on both sides, since both $x$ and $C$ have out-degree $k$. Hence equality holds termwise, so $|\delta_x(D)|$ is independent of the choice of $x\in C$. Thus $\rho$ is lumpable.

Thus, in the category of $k$-out graphs, congruences and lumpable partitions are the same notion. We write $\Cong(G)$ for the set of congruences of $G$. It is a lattice with minimum the identity relation $\Delta_G$ and maximum the universal relation $\nabla_G$. Any congruence different from $\Delta_G$ and $\nabla_G$ is called \emph{nontrivial}.

\begin{remark}\label{rem:quotient-homomorphism-choice}
If $\rho$ is a congruence on a $k$-out graph $G$, then the quotient $G/\rho$ is again $k$-out. The induced map on vertices, $q_v:V(G)\to V(G/\rho)$, $v\mapsto [v]_\rho$, is canonical. The induced map on edges is, in general, not canonical when parallel edges occur in the quotient. However, by lumpability one may always choose a map $q_e:E(G)\to E(G/\rho)$ so that $q=(q_v,q_e):G\to G/\rho$ is a $k$-out homomorphism.

Indeed, for each pair of classes $[u]_\rho,[w]_\rho$, the number $m_{[u],[w]}:=|\delta_x([w]_\rho)|$ is independent of the choice of $x\in [u]_\rho$. Realize the quotient $G/\rho$ with exactly $m_{[u],[w]}$ distinct parallel edges $e^{[u],[w]}_1,\dots,e^{[u],[w]}_{m_{[u],[w]}}$ from $[u]_\rho$ to $[w]_\rho$. Then, for each vertex $x\in [u]_\rho$ and each target class $[w]_\rho$, choose a bijection
\[
\beta_{x,[w]_\rho}:\delta_x([w]_\rho)\longrightarrow \{e^{[u],[w]}_1,\dots,e^{[u],[w]}_{m_{[u],[w]}}\}.
\]
Define $q_e$ by $q_e(e):=\beta_{x,[w]_\rho}(e)$ for $e\in \delta_x([w]_\rho)$. Doing this for every target class $[w]_\rho$ and every vertex $x$ defines $q_e$ on the outgoing star of $x$; since $\sum_{[w]_\rho} m_{[u],[w]}=k$, this yields a bijection from the outgoing star of $x$ onto the outgoing star of $[x]_\rho$. Thus $q$ is a $k$-out homomorphism. In what follows, whenever we speak of the quotient homomorphism $G\to G/\rho$, we mean such a choice.
\end{remark}

In the Markov chain literature, lumpability is typically allowed to include the one-block partition $\{V\}$, equivalently the universal relation $\nabla_G$, which is therefore a valid, albeit degenerate, lumping; see, e.g., \cite[\S6.3]{KemenySnellFMC}. Moreover, computing the coarsest lumping quotient, or more generally the coarsest lumping refining a prescribed initial partition, can be done in polynomial time, for instance in $O(mn)$ time~\cite{Buchholz1994} and even in $O(m\log n)$ time via Paige--Tarjan style refinement~\cite{DerisaviHermannsSanders2003}, where $n$ is the number of states and $m$ the number of nonzero transitions. In Section~\ref{sec: complexity} we shall discuss the complexity of deciding whether a $k$-out graph admits a nontrivial congruence.

\subsection{Partitionability of the Perron--Frobenius eigenvector and synchronization}
Recall that a $k$-out graph $G$ is \emph{totally synchronizing} if every coloring $\chi\in\mathcal C(G)$ yields a synchronizing automaton $\chi(G)$. Since total synchronization is the main topic of the present paper, we now recall a useful criterion based on the Perron--Frobenius eigenvector. Throughout this subsection, we assume that $G$ is strongly connected, with adjacency matrix $A$. This hypothesis implies that $A$ is irreducible. Then $A\mathbf 1 = k\,\mathbf 1$, so the spectral radius of $A$ is $k$. By the Perron--Frobenius theorem for irreducible nonnegative matrices, there exists an entrywise positive left eigenvector $\pi$ such that $\pi^{\top}A = k\,\pi^{\top}$, and it is unique up to multiplication by a positive scalar; see, for instance, \cite{Seneta}. Since $A$ has integer entries, the eigenspace corresponding to the eigenvalue $k$ contains a nonzero rational vector. After scaling by the least common multiple of the denominators, we may therefore assume $\pi\in \mathbb N_{>0}^{V(G)}$. We call the normalization with $\gcd(\pi(v):v\in V(G))=1$ the \emph{integer-normalized} left Perron--Frobenius eigenvector of $G$. In what follows, when we refer to the (left) Perron--Frobenius eigenvector of $G$, we always mean this integer-normalized representative.

\begin{definition}[Partitionable vector]
Let $\mathbf w=(w_1,\dots,w_n)\in\mathbb N^n$. The \emph{$\mathbf w$-weight} of a subset $S\subseteq [n]:=\{1,\dots,n\}$ is
\[
\mathrm{Wg}_{\mathbf w}(S):=\sum_{i\in S} w_i.
\]
We say that $\mathbf w$ is \emph{equipartitionable}, or simply \emph{partitionable}, if there exists a partition
\[
[n]=S_1\sqcup\cdots\sqcup S_m,\qquad m\ge 2,
\]
such that $\mathrm{Wg}_{\mathbf w}(S_i)=\mathrm{Wg}_{\mathbf w}(S_j)$ for all $i,j\in [m]$. In this case we call $\{S_1,\dots,S_m\}$, or equivalently the associated equivalence relation, an \emph{equipartition} of $\mathbf w$. The set $\mathcal P_E(\mathbf w)$ of all equipartitions of $\mathbf w$, ordered by refinement, forms a lattice.
\end{definition}

The previous definition is relevant because partitionability of the Perron--Frobenius eigenvector of a $k$-out graph is closely related to synchronization properties of its automaton colorings. This connection was proved by Friedman~\cite{Fr1990} and slightly extended in \cite{GuPri}; we state here a convenient special case. Recall that for a deterministic automaton $\mathcal A=(Q,\Sigma,\delta)$, a subset $S\subseteq Q$ is called synchronizable if there exist a word $u\in\Sigma^*$ and a state $p\in Q$ such that $S\cdot u=\{p\}$.

\begin{theorem}[\cite{Fr1990,GuPri}]\label{thm:friedman-weight}
Let $G$ be a strongly connected $k$-out graph, and let $\mathbf w\in\mathbb N_{>0}^{V(G)}$ be a left Perron--Frobenius eigenvector. Then, for every coloring $\chi$ of $G$, there exists a partition of $V(G)$ into synchronizing subsets $S_1,\dots,S_m$ with respect to the automaton $\chi(G)$ such that
\[
\mathrm{Wg}_{\mathbf w}(S_1)=\cdots=\mathrm{Wg}_{\mathbf w}(S_m),
\]
and this common value is maximal among the weights of synchronizing subsets. Moreover, the equivalence relation associated with the partition $\{S_1,\dots,S_m\}$ is the kernel of some word.
\end{theorem}

As an immediate consequence, if the integer-normalized left Perron--Frobenius eigenvector of a strongly connected $k$-out graph $G$ is not equipartitionable, then $G$ is totally synchronizing.

Another useful consequence is the following. We recall that a Eulerian $k$-out graph is a $k$-out graph with $|\delta_V(u)|=|\delta_u(V)|=k$ for all $u\in V=V(G)$. 

\begin{proposition}\label{lem:prime-eulerian-dichotomy}
Let $G$ be a strongly connected Eulerian $k$-out graph with $|V(G)|$ prime. Then, for every coloring $\chi$, the automaton $\chi(G)$ is not synchronizing if and only if it is permutative, i.e.\ every letter acts as a permutation of $V(G)$.
\end{proposition}

\begin{proof}
If $\chi(G)$ is permutative, then every word acts as a permutation, so no word can be synchronizing. Conversely, assume that $\chi(G)$ is not permutative. Then some letter collapses a pair of states, hence $\chi(G)$ admits a synchronizing subset of cardinality at least $2$. Since $G$ is Eulerian, its integer-normalized left Perron--Frobenius eigenvector is $\mathbf w=\mathbf 1$. By Theorem~\ref{thm:friedman-weight}, there exists a partition of $V(G)$ into synchronizing subsets of equal maximal $\mathbf w$-weight. As $\mathbf w=\mathbf 1$, these subsets all have the same cardinality $m\ge 2$. If there are $t$ such subsets, then $|V(G)|=tm$, thus by the primality of $|V(G)|$ we conclude $m=|V(G)|$, that is $\chi(G)$ is synchronizing. 
\end{proof}
The following lemma will be used later. 
\begin{lemma}\label{lem:pf-dominating-coordinate}
Let $G$ be a strongly connected $k$-out graph, and let $\pi\in\mathbb N_{>0}^{V(G)}$ be its integer-normalized left Perron--Frobenius eigenvector. If there exists a vertex $v\in V(G)$ such that
\[
\pi(v)>\frac12\sum_{x\in V(G)}\pi(x),
\]
then $\pi$ is not equipartitionable.
\end{lemma}

\begin{proof}
Suppose that $\pi$ were equipartitionable. Then there would exist a partition of $V(G)$ into $m\ge 2$ blocks $S_1, \ldots, S_m$ of equal weight $N=\mathrm{Wg}_{\mathbf w}(S)$. Hence
\[
N=\frac1m\sum_{x\in V(G)}\pi(x)\le \frac12\sum_{x\in V(G)}\pi(x).
\]
But the block containing $v$ has weight at least $\pi(v)$, which is strictly larger than $N$, a contradiction. Therefore $\pi$ is not equipartitionable.
\end{proof}

It is well known that Perron--Frobenius weights aggregate along lumpable partitions; in the Markov-chain language, this is the standard aggregation property of stationary measures for strongly lumpable chains.

\begin{lemma}\label{lem:pf-lump-part}
Let $G$ be a strongly connected $k$-out graph, let $\rho=\{B_1,\dots,B_m\}$ be a lumpable partition, and let $\widehat G:=G/\rho$ be the quotient graph. Let $\mathbf w\in\mathbb N^{V(G)}_{>0}$ and $\widehat{\mathbf w}\in\mathbb N^{[m]}_{>0}$ be the integer-normalized left Perron--Frobenius eigenvectors of $G$ and $\widehat G$, respectively. Then there exists $\alpha\in\mathbb N_{>0}$ such that, for every $r\in[m]$,
\[
\sum_{u\in B_r}\mathbf w(u)=\alpha\,\widehat{\mathbf w}(r).
\]
In particular, if $\widehat{\mathbf w}$ admits an equipartition into blocks of weight $N$, then $\mathbf w$ admits an equipartition, by unions of $\rho$-blocks, into blocks of weight $\alpha N$.
\end{lemma}
\begin{proof}
Let $R$ be the block-incidence matrix of $\rho$. Lumpability is equivalent to $AR=R\widehat A$. Set $\mathbf s^\top:=\mathbf w^\top R$. Then
$\mathbf s^\top\widehat A=\mathbf w^\top R\widehat A=\mathbf w^\top AR=k\,\mathbf w^\top R=k\,\mathbf s^\top$.
Thus $\mathbf s$ is a positive integral left eigenvector of $\widehat A$ for the eigenvalue $k$. By Perron--Frobenius, $\mathbf s$ is proportional to the integer-normalized vector $\widehat{\mathbf w}$, so $\mathbf s=\alpha\,\widehat{\mathbf w}$ for some $\alpha\in\mathbb N_{>0}$. This is exactly the claimed identity. The final statement follows by grouping $\rho$-blocks.
\end{proof}

\section{Totally synchronizing $k$-out  and their symmetries}

Let $G$ be a $k$-out graph, let $\chi\in\mathcal C(G)$ be a coloring, and view $\chi(G)$ as an automaton. We denote by  $\Aut(\chi(G))$ (resp. $\End(\chi(G))$) the automorphisms (resp. the endomorphisms) that preserve the coloring in the following sense. 
An automorphism $\varphi\in \Aut(G)$ is \emph{compatible} with the coloring $\chi$ if $\chi(\varphi_e(e))=\chi(e)$ for all edges $e\in E(G)$. 
If $\delta_\chi$ denotes the transition function of $\chi(G)$, then compatibility is equivalent to
\[
\varphi_v(\delta_\chi(v,a))=\delta_\chi(\varphi_v(v),a)
\qquad (v\in V(G),\ a\in[k]),
\]
that is, $\varphi_v\in \Aut(\chi(G))$.

Note that any endomorphism $\varphi \in \End(\chi(G))$ of the colored automaton induces, by forgetting the colors, an endomorphism of the underlying
$k$-out directed graph. By abuse of notation, we denote this induced
endomorphism again by $\varphi \in \End(G)$. This defines a natural injective
homomorphism
\[
\End(\chi(G)) \hookrightarrow \End(G),
\]
and hence we may regard $\End(\chi(G))$ as a submonoid of $\End(G)$.

\begin{lemma}\label{lem: nec condition compatibility}
Let $G$ be a strongly connected $k$-out graph, and let
$\varphi \in \Aut(G)$ be a graph automorphism with $\varphi \neq \id$
that fixes a vertex, i.e.,
\[
\exists\, v \in V(G) \text{ such that } \varphi(v) = v .
\]
Then $\varphi$ is not compatible with any automaton coloring $\chi$ on $G$.
\end{lemma}
\begin{proof}
First note that if $\varphi_v = \id$, then necessarily $\varphi_e$ swaps two
distinct edges $e_1,e_2$ with the same source and target, that is,
\[
s(e_1)=s(e_2)=v
\quad\text{and}\quad
t(e_1)=t(e_2)=v'.
\]
However, $\varphi$ cannot be compatible with any automaton coloring, since
compatibility would imply $\chi(e_1)=\chi(e_2)$, contradicting the definition
of an automaton coloring. Therefore, we may assume $\varphi_v \neq \id$.
Let $\Fix(\varphi)$ denote the set of vertices of $G$ fixed by $\varphi$.
Since $G$ is strongly connected and $\varphi_{v} \neq \id$, there exists a vertex
$p \in \Fix(\varphi)$ adjacent to some vertex $p' \notin \Fix(\varphi)$, that is, $\varphi(p')=p''\neq p'$.
Let $e,f \in E(G)$ be edges such that
\[
s(e)=s(f)=p,
\qquad
t(e)=p',
\qquad
t(f)=p'',
\]
and $f=\varphi(e)$. Note that $e \neq f$ since $p' \neq p''$. Now assume that $\chi$ is an automaton coloring compatible with $\varphi$.
Then
\[
\varphi\bigl(p \xrightarrow{\chi(e)} p'\bigr)
=
\left(\varphi(p) \xrightarrow{\chi(\varphi(e))} \varphi(p')\right)
=
\left(p \xrightarrow{\chi(f)} p''\right),
\]
and compatibility implies $\chi(f)=\chi(e)$. This yields two distinct edges
starting at $p$ with the same color, contradicting the definition of an
automaton coloring.
\end{proof}

In view of the previous lemma, we introduce the following definition.

\begin{definition}[vertex-rigid]
An non-trivial endomorphism $\varphi \in \End(G)$ is said to be \emph{vertex-rigid at} $v\in V(G)$ if, for every $m>0$,
\[
\varphi^{m}(v)=v
\quad\Longrightarrow\quad
\varphi^{m}(e)=e
\quad \text{for all edges } e \in \delta_v(V(G))
\] 
We say that $\varphi$ is \emph{vertex-rigid on a subset} $V'\subseteq V(G)$ if for every vertex $u\in V'$, $\varphi$ is vertex-rigid at $u$. In the case $V'=V(G)$, we simply say that $\varphi$ is \emph{vertex-rigid}. The set of vertex-rigid endomorphism (automorphism) is denoted by $\REnd(G)$ ($\RAut(G)$). 
\end{definition}
We recall the following terminology. For an endomorphism $\varphi\in\End(G)$, set
\[
\Fix_V(\varphi):=\{v\in V(G)\mid \varphi_v(v)=v\}.
\]
A nontrivial automorphism $\varphi\in\Aut(G)$ is \emph{semiregular} if every non-identity power of $\varphi$ has no fixed vertex, i.e.\
\[
\varphi^m\neq \id_G \quad\Longrightarrow\quad \Fix_V(\varphi^m)=\emptyset
\qquad (m>0).
\]
Equivalently, every orbit of the cyclic group $\langle\varphi\rangle$ on $V(G)$ has cardinality the order $\operatorname{ord}(\varphi)$. 

We have the following fact.

\begin{proposition}\label{prop:vertex-rigid-semiregular}
Let $G$ be a strongly connected $k$-out graph. Then $\End(G)=\Aut(G)$.
Moreover, a nontrivial endomorphism $\varphi\in\End(G)$ is vertex-rigid if and only if it is a semiregular automorphism.
\end{proposition}

\begin{proof}
By Proposition~\ref{prop:strongly-connected-endo-auto}, every endomorphism of $G$ is an automorphism. Let $\varphi\in\End(G)=\Aut(G)$ be nontrivial. Assume first that $\varphi$ is vertex-rigid.
Suppose, by contradiction, that $\varphi$ is not semiregular. Then there exists $m>0$ such that
$\psi:=\varphi^m$ is not the identity and fixes at least one vertex. Let $F:=\FixV{\psi}$. If $F=V(G)$, then $\psi$ fixes every vertex. Since $\psi\neq\id_G$, it moves some edge $e$, contradicting vertex-rigidity. If $F\subsetneq V(G)$, then $F$ is nonempty and proper. For every $v\in F$, vertex-rigidity implies
that $\psi$ fixes the outgoing star of $v$ pointwise. Hence every out-neighbour of every vertex of
$F$ lies again in $F$. Thus $F$ is closed under outgoing edges, contradicting strong connectivity.
Therefore $\varphi$ is semiregular.

Conversely, assume that $\varphi$ is semiregular. If $\varphi^m$ fixes a vertex for some $m>0$, then
$\varphi^m=\id_G$. Hence $\varphi^m$ fixes every outgoing star pointwise. Thus $\varphi$ is
vertex-rigid.
\end{proof}

\begin{proposition}\label{prop:semiregular-vertex-rigid}
Let $G$ be a $k$-out graph. Every semiregular automorphism of $G$ is vertex-rigid.
\end{proposition}

\begin{proof}
Let $\varphi\in\Aut(G)$ be semiregular. Let $m>0$ and suppose that $\varphi^m$ fixes a vertex.
By semiregularity, $\varphi^m=\id_G$. Hence $\varphi^m$ fixes every outgoing edge of every vertex.
Thus $\varphi$ is vertex-rigid.
\end{proof}
\begin{proposition}\label{pro: synch and automorphisms}
Let $\mathcal A=(Q,\Sigma,\delta)$ be a synchronizing automaton, and let $\varphi\in\End(\mathcal A)$.
If $u\in\Sigma^*$ is a synchronizing word and $Q\cdot u=\{p\}$, then $\varphi(p)=p$.
In particular:
\begin{itemize}
\item if $\mathcal A$ is strongly connected, then $\End(\mathcal A)=\{\id\}$;
\item if $\mathcal A$ admits a fixed-point-free automorphism, i.e.\ a derangement of $Q$, then $\mathcal A$ is not synchronizing.
\end{itemize}
\end{proposition}
\begin{proof}
Since $\varphi$ is an endomorphism, it commutes with the action of words. Hence
\[
\varphi(p)
=
\varphi(Q\cdot u)
=
\varphi(Q)\cdot u
\subseteq
Q\cdot u
=
\{p\}.
\]
Thus $\varphi(p)=p$. Assume now that $\mathcal A$ is strongly connected. Let $q\in Q$. Choose $v\in\Sigma^*$ such that
$p\cdot v=q$. Then
\[
\varphi(q)
=
\varphi(p\cdot v)
=
\varphi(p)\cdot v
=
p\cdot v
=
q.
\]
Hence $\varphi=\id$.
Finally, suppose that $\psi\in\Aut(\mathcal A)$ is fixed-point-free. If $\mathcal A$ were synchronizing,
the first part applied to $\psi$ would give a state fixed by $\psi$, a contradiction. Therefore
$\mathcal A$ is not synchronizing.
\end{proof}
Note that if $\chi$ is an automaton coloring of a strongly connected $k$-out graph $G$, then the automaton $\chi(G)$ is also strongly connected. We now study the existence of endomorphisms compatible with a coloring.

\begin{lemma}\label{lem:compatible-coloring-partial}
Let $G$ be a $k$-out directed graph and let $\varphi\in \End(G)$. Fix $v\in V(G)$, and let $B_\varphi(v)$ be its basin and $C_v\subseteq B_\varphi(v)$ the corresponding limit cycle. Assume that $\varphi$ is vertex-rigid at every vertex of $C_v$. Then, for every permutation $\sigma\in S_k$, there exists a partial coloring
\[
\chi_{\varphi,v}(\sigma):\delta_{B_\varphi(v)}(V(G))\to [k]
\]
such that $\chi_{\varphi,v}(\sigma)(\varphi(e))=\chi_{\varphi,v}(\sigma)(e)$ for every $e\in \delta_{B_\varphi(v)}(V(G))$.
\end{lemma}

\begin{proof}
Choose a vertex $p\in C_v$, and let $d\ge 1$ be its period under $\varphi$. Let $e_1,\dots,e_k$ be the $k$ edges with source $p$, and set $\chi_{\varphi,v}(\sigma)(e_i):=\sigma(i)$ for $i=1,\dots,k$. If $u\in C_v$ and $e$ is an edge with source $u$, then, since $\varphi$ acts cyclically on $C_v$ and is locally bijective on outgoing stars, there exist $m\ge 0$ and $i\in\{1,\dots,k\}$ such that $\varphi^m(p)=u$ and $\varphi^m(e_i)=e$. Define $\chi_{\varphi,v}(\sigma)(e):=\sigma(i)$. This is well defined: if also $\varphi^{m'}(e_j)=e$, then $m'-m$ is a multiple of $d$, say $m'=m+td$. Since $\varphi^d(p)=p$ and $\varphi$ is vertex-rigid at $p$, the map $\varphi^d$ fixes every edge with source $p$, hence $\varphi^{td}(e_i)=e_i$. Therefore $e_j=e_i$ and so $\sigma(j)=\sigma(i)$. Now let $u\in B_\varphi(v)\setminus C_v$, and let $e$ be an edge with source $u$. Choose $m\ge 0$ minimal such that $\varphi^m(u)\in C_v$, and define
\[
\chi_{\varphi,v}(\sigma)(e):=\chi_{\varphi,v}(\sigma)(\varphi^m(e)).
\]
This is well defined because the minimal such $m$ is unique. By construction, $\chi_{\varphi,v}(\sigma)(\varphi(e))=\chi_{\varphi,v}(\sigma)(e)$ whenever both sides are defined. Hence $\chi_{\varphi,v}(\sigma)$ is a partial coloring of $\delta_{B_\varphi(v)}(V(G))$ compatible with $\varphi$.
\end{proof}

We have the following lemma.

\begin{lemma}\label{lem:compatible-coloring-full}
Let $\varphi\in \End(G)$ be a vertex-rigid endomorphism, and let $B_\varphi(v_1),\dots,B_\varphi(v_\ell)$ be the partition of $V(G)$ into $\varphi$-basins. For each basin $B_\varphi(v_j)$ fix a permutation $\sigma_j\in S_k$. Then there exists an automaton coloring $\chi_\varphi(\sigma_1,\dots,\sigma_\ell)$ compatible with $\varphi$.
\end{lemma}

\begin{proof}
Since $\varphi$ is vertex-rigid, Lemma~\ref{lem:compatible-coloring-partial} applies to each basin $B_\varphi(v_j)$ and yields a partial coloring
\[
\chi_{\varphi,v_j}(\sigma_j):\delta_{B_\varphi(v_j)}(V(G))\to [k]
\]
compatible with $\varphi$. As the basins form a partition of $V(G)$, the sets $\delta_{B_\varphi(v_j)}(V(G))$ are pairwise disjoint and their union is $E(G)$. Hence the partial colorings glue together to a unique coloring $\chi_\varphi(\sigma_1,\dots,\sigma_\ell):E(G)\to [k]$, which is an automaton coloring and is compatible with $\varphi$.
\end{proof}

\begin{proposition}\label{prop: rigid in end of automaton}
Let $\mathcal{C}(G)$ denote the set of (automata) colorings of $G$. Then
\[
\REnd(G) = \bigcup_{\chi \in \mathcal{C}(G)} \End(\chi(G)).
\]
\end{proposition}

\begin{proof}
The inclusion of the right-hand side in the left-hand side follows directly from the embeddings
\(\End(\chi(G)) \hookrightarrow \REnd(G)\) for each \(\chi \in \mathcal{C}(G)\).  For the reverse inclusion, note that by Lemma~\ref{lem:compatible-coloring-full}, for any vertex-rigid endomorphism
\(\varphi \in \REnd(G)\) and any fixed tuple \(\sigma_1, \dots, \sigma_\ell \in S_k\), there exists a coloring
\(\chi_{\varphi}(\sigma_1, \dots, \sigma_\ell)\) that is compatible with \(\varphi\). 
That is, \(\varphi\) is also an endomorphism of the colored graph
\(\chi_{\varphi}(\sigma_1, \dots, \sigma_\ell)(G)\).
\end{proof}
We have the following theorem.
\begin{theorem}\label{theo: main}
Let $G$ be a strongly connected $k$-out totally synchronizing graph.
Then $\Aut(G)$ contains no vertex-rigid automorphism, in particular, it has no semiregular elements.
\end{theorem}

\begin{proof}
Let $\varphi\in\Aut(G)$ be semiregular. Then $\varphi$ is vertex-rigid.
By Proposition~\ref{prop: rigid in end of automaton}, there exists a coloring $\chi$ of $G$ such that
$\varphi\in\End(\chi(G))$.
Since $G$ is strongly connected and totally synchronizing, $\chi(G)$ is strongly connected and synchronizing.
By Proposition~\ref{pro: synch and automorphisms}, $\End(\chi(G))=\{\id\}$.
Thus $\varphi=\id$, contradiction.
\end{proof}
The following observation isolates the mechanism behind the case $k=2$, $p=3$ announced by Pribavkina~\cite{PribavkinaSandGAL2019}.

\begin{lemma}\label{lem:prime-order-derangement}
Let $G$ be a strongly connected $k$-out graph, and let $\varphi\in\Aut(G)$ have prime order $p>k$. Then either $\varphi=\id$ or $\phi$ is a derangement.
\end{lemma}

\begin{proof}
Suppose that $\varphi$ fixes a vertex $v$. Since $\varphi$ is an automorphism, it permutes the $k$ outgoing edges of $v$. As $\varphi$ has order $p$, the induced permutation on the outgoing star of $v$ has order dividing $p$. Since $p$ is prime, every orbit has size either $1$ or $p$. But $p>k$, so no orbit can have size $p$. Hence every outgoing edge of $v$ is fixed by $\varphi$. If $e$ is such an edge, then \(\varphi_v(t(e))=t(\varphi_e(e))=t(e),\) so the target of $e$ is also fixed. Therefore the set of fixed vertices of $\varphi$ is closed under outgoing edges. Since $G$ is strongly connected, this set must be either empty or all of $V(G)$. As it contains $v$, every vertex is fixed.

Applying the same argument at each vertex, we conclude that every outgoing edge of every vertex is fixed. Hence $\varphi=\id$.
\end{proof}

As an immediate consequence of Theorem~\ref{theo: main}, we obtain the following corollary, which extends the case $k=2$, $p=3$ announced in~\cite{PribavkinaSandGAL2019}.

\begin{corollary}\label{cor:prime-divisor-greater-than-k}
Let $G$ be a strongly connected $k$-out graph. If $|\Aut(G)|$ is divisible by a prime $p>k$, then $G$ is not totally synchronizing.
\end{corollary}

\begin{proof}
By Cauchy's theorem, $\Aut(G)$ contains an element $\phi$ of order $p$. By Lemma~\ref{lem:prime-order-derangement}, $\phi$ is a derangement. Since $\phi$ has prime order $p$, every nontrivial cycle of $\phi$ has length $p$. Therefore $\phi$ is semiregular. By Theorem~\ref{theo: main}, $G$ is not totally synchronizing.
\end{proof}

It is interesting to note the connection with a classical conjecture of Marušič~\cite{Marusic1981}, also called the Polycirculant Conjecture, which asserts that every finite vertex-transitive (di)graph admits a nontrivial semiregular automorphism. This conjecture imposes vertex-transitivity on the structure of the automorphism group. However, it does not apply in our setting, since totally synchronizing graphs are never vertex-transitive. In a private communication and in a conference talk~\cite{PribavkinaSandGAL2019}, Gusev and Pribavkina observed the following.
\begin{proposition}
Let $G$ be a strongly connected $k$-out graph. If $\End(G)$ is vertex transitive on $V(G)$, then $G$ is not totally synchronizing.
\end{proposition}

\begin{proof}
Since $G$ is strongly connected, Proposition~\ref{prop:strongly-connected-endo-auto} gives $\End(G)=\Aut(G)$. Thus $\Aut(G)$ is vertex transitive. Let $p,q\in V(G)$. Choose $\varphi\in\Aut(G)$ such that $\varphi(p)=q$. Since automorphisms preserve incoming multiplicities, we have
\(
|\delta_V(p)|=|\delta_{\varphi_v(V)}(\varphi_v(p))|=|\delta_V(q)|.
\)
Hence all vertices have the same in-degree, say $d$. Since $G$ is $k$-out, we have 
\[
k|V|=|E|=\sum_{v\in V} |\delta_V(v)|=d|V|.
\]
Therefore $d=k$, so every vertex has in-degree $k$ as well. Hence $G$ is Eulerian. By~\cite[Lemma~1]{Kari2003}, the graph $G$ is not totally synchronizing.
\end{proof}
We end this subsection with two illustrative $2$-out examples (Figure~\ref{fig:examples}) exhibiting nontrivial vertex-fixing automorphism groups of the kind described in Theorem~\ref{theo: main}.
In the left-hand Eulerian graph, $\Aut(G)$ is generated by the automorphism sending $1\mapsto 1'$, $2\mapsto 2'$ and fixing $0$.
Since the graph is Eulerian, it admits a non-synchronizing coloring by~\cite[Lemma~1]{Kari2003}, so it is not totally synchronizing.
The right-hand graph still has a nontrivial automorphism (the transposition $(1\ 2)$ fixing $0,3,4,5$), but it is totally synchronizing since its
integer-normalized left Perron--Frobenius eigenvector $\pi=(4,3,3,6,3,2)$ is not equipartitionable. To see that $\pi=(4,3,3,6,3,2)$ is not equipartitionable, note that $\sum_{v}\pi(v)=21$. Thus any nontrivial equipartition into $m\ge2$ blocks would require $m\mid 21$, hence $m=3$, and each block would have weight $7$. But the only way to obtain $7$ from the multiset $\{6,4,3,3,3,2\}$ is $4+3$, which cannot produce three disjoint blocks of weight $7$.
Moreover, this last example shows that Corollary~\ref{cor:prime-divisor-greater-than-k} is sharp: the prime $p=2$ divides $|\Aut(G)|=2$ but $p\not>k$.

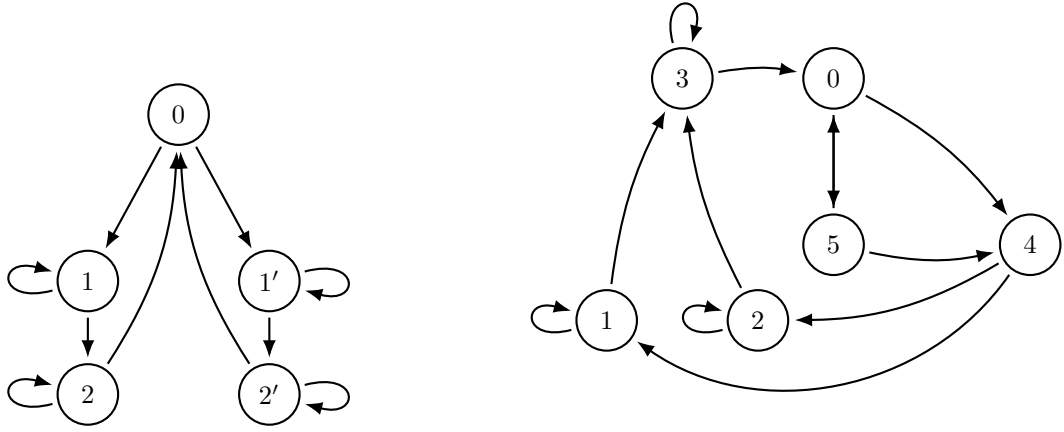
\begin{figure}[ht]
\centering
\begin{subfigure}[t]{0.5\textwidth}
\centering
\begin{tikzpicture}[
  >=Latex,
  node/.style={circle, draw, thick, minimum size=8mm, inner sep=1pt},
  edge/.style={->, thick, shorten <=2pt, shorten >=2pt}
]
\node[node] (v0) at (0,3.7) {$0$};
\node[node] (v1) at (-1.2,1.5) {$1$};
\node[node] (v1p) at (1.2,1.5) {$1'$};
\node[node] (v2) at (-1.2,0) {$2$};
\node[node] (v2p) at (1.2,0) {$2'$};

\draw[edge] (v0) -- (v1);
\draw[edge] (v0) -- (v1p);

\draw[edge] (v1) -- (v2);
\draw[edge] (v1p) -- (v2p);

\draw[edge] (v2) to[bend right=15] (v0);
\draw[edge] (v2p) to[bend left=15] (v0);

\draw[edge] (v1)  edge[loop left,  looseness=10] (v1);
\draw[edge] (v1p) edge[loop right, looseness=10] (v1p);
\draw[edge] (v2)  edge[loop left,  looseness=10] (v2);
\draw[edge] (v2p) edge[loop right, looseness=10] (v2p);
\end{tikzpicture}
\label{fig:eulerian-fixed-vertex}
\end{subfigure}
\hfill
\begin{subfigure}[t]{0.49\textwidth}
\centering
\begin{tikzpicture}[
  >=Latex,
  node/.style={circle, draw, thick, minimum size=8mm, inner sep=1pt},
  edge/.style={->, thick, shorten <=2pt, shorten >=2pt}
]
\node[node] (v0) at (0, 2.2) {$0$};
\node[node] (v5) at (0,0) {$5$};
\node[node] (v1) at (-3, -1) {$1$};
\node[node] (v2) at (-1, -1) {$2$};
\node[node] (v3) at (-2, 2.2) {$3$};
\node[node] (v4) at ( 2.6, 0) {$4$};

\draw[edge] (v1) edge[loop left, looseness=9] (v1);
\draw[edge] (v2) edge[loop left, looseness=9] (v2);
\draw[edge] (v3) edge[loop above, looseness=9] (v3);

\draw[edge] (v0) to[bend left=12] (v4);
\draw[edge] (v0) -- (v5);

\draw[edge] (v5) -- (v0);
\draw[edge] (v5) to[bend right=12] (v4);

\draw[edge] (v4) to[bend left=15] (v2);
\draw[edge] (v4) to[bend left=45] (v1);

\draw[edge] (v1) to[bend left=10] (v3);

\draw[edge] (v2) to[bend left=10] (v3);

\draw[edge] (v3) to[bend left=10] (v0);

\end{tikzpicture}
\end{subfigure}

\caption{Two $2$-out examples with a non-regular automorphism: an Eulerian graph (left) and a totally synchronizing graph (right).}
\label{fig:examples}
\end{figure}

\subsection{Construction of $k$-out graphs with prescribed quotient and automorphism group}

The following construction is inspired by the theory of voltage and covering graphs, combined with a Frucht-type rigidification; see, for instance, \cite[Ch.~2]{GrossTuckerTGT}. In our setting, the aim is to prescribe a finite automorphism group while keeping a given $k$-out quotient fixed. The point is that the lift construction naturally produces nontrivial fixed-point-free deck transformations. These transformations are vertex-rigid, hence they become automorphisms of suitable colorings; by Proposition~\ref{pro: synch and automorphisms}, the resulting graphs are not totally synchronizing.
The construction is in the spirit of the classical Frucht argument, but the present setting is substantially more constrained: one must preserve simultaneously the $k$-out condition and the blockwise edge multiplicities defining the prescribed quotient. Thus one cannot freely attach arbitrary asymmetric gadgets, and the rigidification requires a more careful analysis.
The main outcome of this subsection is that, after a suitable rigid refinement of the quotient and a compatible voltage assignment, one obtains a strongly connected $k$-out graph $G$ whose quotient is the prescribed graph (up to some added loops) and whose automorphism group, modulo the vertex-trivial automorphisms coming from parallel edges, is isomorphic to a given finite group $H$. In particular, this provides a broad family of strongly connected $k$-out graphs which are not totally synchronizing with a prescribed automorphism group. The first step is proving the following ``rigid refinement" construction.
\begin{proposition}\label{prop:vertex-rigid-refinement-permutation}
Let $G=(V,E,s,t)$ be a finite strongly connected $k$-out graph, with $k\ge 2$.
Assume that there exists a permutation $\tau:V\to V$ such that $|\delta_u(\tau(u))|\ge 1$ for every $u\in V$, and such that the graph $G^-=(V,E^-,s,t)$, obtained by removing exactly one edge from $\delta_u(\tau(u))$ for each $u\in V$, is still strongly connected ($(k-1)$-out graph). Then there exist a finite strongly connected $k$-out graph $\widetilde G=(\widetilde V,\widetilde E,s,t)$ and a congruence $\rho_0$ on $\widetilde G$ such that $\widetilde G/\rho_0 \simeq G$, 
and every automorphism of $\widetilde G$ fixes every vertex. Equivalently, $\Aut(\widetilde G)=\Aut_E(\widetilde G)$.
\end{proposition}

The proof is slightly longer than a standard rigidification argument, because the construction must preserve simultaneously the $k$-out condition and the blockwise edge multiplicities defining the quotient. For these reasons, we present here a detailed proof. 
\begin{figure}[ht]
\centering
\begin{tikzpicture}[
  >=stealth,
  block/.style={draw, rounded corners, minimum width=2.0cm, minimum height=1.1cm},
  vtx/.style={circle, draw, fill=black, inner sep=1.2pt},
  aux/.style={circle, draw, inner sep=1.2pt},
  every node/.style={font=\small}
]

\node[block,label=above:$B_u$] (B0) at (0,0) {};
\node[block,label=above:$B_{\tau(u)}$] (B1) at (3,0) {};
\node[block,label=above:$B_{\tau^2(u)}$] (B2) at (6,0) {};
\node at (8,0) {$\cdots$};
\node[block,label=above:$B_{\tau^{L_u}(u)}$] (BL) at (10.5,0) {};
\node[block,label=above:$B_{\tau^{L_u+1}(u)}$] (BF) at (14,0) {};

\node[vtx,label=below:$r_u$] (ru) at (-0.45,0) {};
\node[aux,label=below:$c_{u,1}$] (c1) at (2.55,0) {};
\node[aux,label=below:$c_{u,2}$] (c2) at (5.55,0) {};
\node[aux,label=below:$c_{u,L_u}$] (cL) at (10.05,0) {};
\node[vtx,label=below:$\quad r_{\tau^{L_u+1}(u)}$] (rf) at (13.55,0) {};

\draw[->,dashed,thick] (ru) -- (c1);
\draw[->,dashed,thick] (c1) -- (c2);
\draw[->,dashed,thick] (c2) -- (7.2,0);
\draw[->,dashed,thick] (8.8,0) -- (cL);
\draw[->,dashed,thick] (cL) -- (rf);

\end{tikzpicture}
\caption{For each $u\in V$, the walk
$r_u\to c_{u,1}\to \cdots \to c_{u,L_u}\to r_{\tau^{L_u+1}(u)}$
places $c_{u,i}$ in the block $B_{\tau^i(u)}$.}
\end{figure}
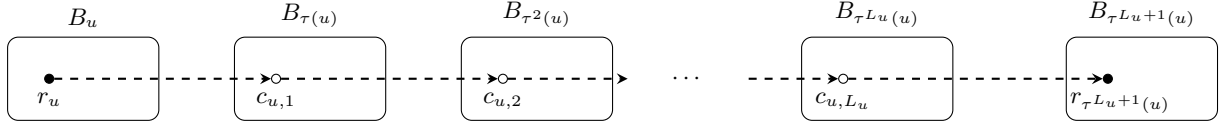
\begin{lemma}
\label{lem:construction-refinement-permutation}
Under the assumptions of Proposition~\ref{prop:vertex-rigid-refinement-permutation}, there exists a finite $k$-out graph $\widetilde G=(\widetilde V,\widetilde E,s,t)$ and a partition
\[
\widetilde V=\bigsqcup_{w\in V} B_w
\]
such that the partition $\{B_u:u\in V\}$ is lumpable, equivalently, the associated equivalence relation $\rho_0$ is a congruence, and $\widetilde G/\rho_0 \simeq G$.
\end{lemma}

\begin{proof}
Set $m(u,w):=|\delta_u(w)|$, and define
\[
m^-(u,w):=
\begin{cases}
m(u,w)-1,& w=\tau(u),\\
m(u,w),& w\neq \tau(u).
\end{cases}
\]
Since $G^-$ is $(k-1)$-out, we have $m^-(u,w)\ge 0$ and $\sum_{w\in V} m^-(u,w)=k-1$ for every $u\in V$. Choose pairwise distinct integers $L_u\ge 1$ for $u\in V$. For each $u\in V$, introduce vertices $r_u,\ c_{u,1},\dots,c_{u,L_u}$, and place them into blocks by declaring $r_u\in B_u$ and $c_{u,i}\in B_{\tau^i(u)}$. Equivalently,
\[
B_w=\{r_w\}\cup\{c_{u,i}:\tau^i(u)=w,\ 1\le i\le L_u\}.
\]
Let $\widetilde V:=\bigsqcup_{w\in V} B_w$. Now define the edges of $\widetilde G$ as follows. For each $u\in V$, add the path
\[
r_u\to c_{u,1}\to c_{u,2}\to \cdots \to c_{u,L_u}\to r_{\tau^{L_u+1}(u)}.
\]
Each such edge will be called a ``coding'' edge (since it is used to ``distinguish" the root $r_u$). Next, for each $v,w\in V$ and each $x\in B_v$, add exactly $m^-(v,w)$ edges from $x$ to $r_w$. These are the ``core'' edges. Each vertex $x\in B_v$ therefore has exactly one coding edge and $\sum_w m^-(v,w)=k-1$ core edges, hence out-degree $k$. So $\widetilde G$ is $k$-out. Fix $v\in V$, $x\in B_v$, and $w\in V$. The number of edges from $x$ to vertices in $B_w$ is
\[
|\delta_x(B_w)|=m^-(v,w)+
\begin{cases}
1,& w=\tau(v),\\
0,& w\neq \tau(v),
\end{cases}
\]
hence $|\delta_x(B_w)|=m(v,w)=|\delta_v(w)|$. Thus all vertices in the same block $B_v$ have the same blockwise outgoing multiplicities. Hence the equivalence relation $\rho_0$ associated with the classes $B_w$ is a congruence and $\widetilde G/\rho_0\simeq G$.
\end{proof}
\begin{lemma}
\label{lem:strong-connectivity-refinement-permutation}
The graph $\widetilde G$ constructed in Lemma~\ref{lem:construction-refinement-permutation} is strongly connected.
\end{lemma}

\begin{proof}
Let $R:=\{r_u:u\in V\}$. For $u,w\in V$, there are exactly $m^-(u,w)$ core edges from $r_u$ to $r_w$, so the subgraph induced on $R$ by the core edges is isomorphic to $G^-$. Hence $R$ is strongly connected. Now let $y\in \widetilde V$. If $y\in R$, there is nothing to prove. If $y=c_{u,i}$, then $r_u$ reaches $y$ by following coding edges, and $y$ reaches the root $r_{\tau^{L_u+1}(u)}$ by continuing along the same coding chain. Since $R$ is strongly connected, every root reaches every other root. Therefore every vertex is reachable from $R$ and reaches $R$, so $\widetilde G$ is strongly connected.
\end{proof}

\begin{lemma}
\label{lem:vertex-rigidity-refinement-permutation}
Every automorphism of $\widetilde G$ fixes every vertex.
\end{lemma}
\begin{proof}
We first characterize the roots intrinsically. Since $\tau$ is a permutation, for every $w\in V$ the block $B_w$ contains at least two vertices: indeed, $r_w\in B_w$ and also $c_{\tau^{-1}(w),1}\in B_w$. Thus $|B_w|\ge 2$. Now fix $w\in V$. Because $G^-$ is strongly connected, $w$ has positive in-degree in $G^-$. Hence there exists $v\in V$ such that $m^-(v,w)\ge 1$. Every vertex of $B_v$ contributes $m^-(v,w)$ core edges to $r_w$, and $|B_v|\ge 2$, so $r_w$ has in-degree at least $2$ in $\widetilde G$. On the other hand, every non-root vertex has in-degree exactly $1$: if $x=c_{u,1}$, the only edge ending at $x$ is $r_u\to c_{u,1}$; if $x=c_{u,i}$ with $i>1$, the only edge ending at $x$ is $c_{u,i-1}\to c_{u,i}$. No core edge ends at a non-root vertex. Since automorphisms are graph isomorphisms, they preserve in-degrees. Hence, the set of roots
\[
R=\{x\in \widetilde V: |\delta_{\widetilde V}(x)|\ge 2\}
\]
is preserved setwise by every automorphism. Set $X:=\widetilde V\setminus R$. Since every automorphism preserves in-degree, it preserves $R$ setwise, hence also $X$. For every root $r\in R$, there is exactly one outgoing edge from $r$ to a vertex of $X$. Indeed, if $r=r_u$, this is precisely the edge $r_u\to c_{u,1}$; all other outgoing edges from $r_u$ end in $R$. Moreover, every vertex of $X$ has at most one outgoing edge to a vertex of $X$. More precisely, if $x=c_{u,i}$ with $1\le i<L_u$, then the unique outgoing edge from $x$ to $X$ is $c_{u,i}\to c_{u,i+1}$, whereas if $x=c_{u,L_u}$, then $x$ has no outgoing edge to $X$. It follows that for every $r\in R$ there exists a unique maximal directed path $r=x_0\to x_1\to \cdots \to x_{\ell(r)}$ such that $x_1,\dots,x_{\ell(r)}\in X$.
For $r_u\in R$, this path is exactly $r_u\to c_{u,1}\to \cdots \to c_{u,L_u}$, hence $\ell(r_u)=L_u$. Now let $\varphi\in\Aut(\widetilde G)$. Since $\varphi$ preserves $R$ and $X$ setwise and preserves adjacency, it sends the unique maximal directed path associated with $r$ onto the unique maximal directed path associated with $\varphi(r)$. In particular, $\ell(\varphi(r))=\ell(r)$ for all $r\in R$.
Since we have chosen the integers $L_u$ pairwise distinct, the roots are distinguished by the values $\ell(r_u)=L_u$. Therefore every automorphism fixes each root $r_u$. Once $r_u$ is fixed, the whole chain
\[
r_u\to c_{u,1}\to \cdots \to c_{u,L_u}\to r_{\tau^{L_u+1}(u)}
\]
is fixed vertexwise: $c_{u,1}$ is the unique out-neighbor of $r_u$ outside $R$, and inductively $c_{u,i+1}$ is the unique out-neighbor of $c_{u,i}$ outside $R$ for $1\le i<L_u$. Since every vertex is either a root or belongs to one of these chains, every automorphism fixes every vertex.
\end{proof}

\begin{proof}[Proof of Proposition~\ref{prop:vertex-rigid-refinement-permutation}]
Lemma~\ref{lem:construction-refinement-permutation} gives the graph $\widetilde G$ and the congruence $\rho_0$, with quotient $\widetilde G/\rho_0\simeq G$. By Lemma~\ref{lem:strong-connectivity-refinement-permutation}, the graph $\widetilde G$ is strongly connected. By Lemma~\ref{lem:vertex-rigidity-refinement-permutation}, every automorphism of $\widetilde G$ fixes every vertex. Hence $\Aut(\widetilde G)=\Aut_E(\widetilde G)$.
\end{proof}

\begin{proposition}\label{prop:parallel-free-rigidification}
Under the assumptions of Proposition~\ref{prop:vertex-rigid-refinement-permutation}, assume moreover that:
\begin{enumerate}
\item for every $v,w\in V$ with $w\neq \tau(v)$, one has $m(v,w)\le 1$;
\item for every $v\in V$, one has $m(v,\tau(v))\le 2$;
\item every $\tau$-cycle contains a vertex $g$ such that $m(g,\tau(g))=1$.
\end{enumerate}
Then the integers $L_u$ can be chosen so that the resulting refinement $\widetilde G$ has no parallel edges. In particular,
\[
\Aut(\widetilde G)=\{\id\}.
\]
\end{proposition}

\begin{proof}
For each $\tau$-cycle $O$, choose a vertex $g_O\in O$ such that $m(g_O,\tau(g_O))=1$, as in (3). For every $u\in O$, choose a positive integer $L_u$ such that $\tau^{L_u}(u)=g_O$. Since the vertices in different $\tau$-cycles are disjoint, and within each cycle one may add sufficiently large multiples of $|O|$, the integers $L_u$ can be chosen pairwise distinct.
With this choice, every terminal vertex $c_{u,L_u}$ lies in a block $B_{g_O}$ such that $m(g_O,\tau(g_O))=1$.
We claim that the resulting graph $\widetilde G$ has no parallel edges. Let $x\in B_v$.
If $w\neq \tau(v)$, then the number of core edges from $x$ to $r_w$ is $m(v,w)$, which is at most $1$ by (1). Hence no parallel edges arise in this case.
Now consider $w=\tau(v)$. The number of core edges from $x$ to $r_{\tau(v)}$ is $m(v,\tau(v))-1$, which is at most $1$ by (2). Thus parallel edges can only occur if the distinguished edge from $x$ also ends at $r_{\tau(v)}$. If $x$ is not a terminal vertex of one of the distinguished walks, then its distinguished edge ends at a non-root vertex, so it cannot be parallel to a core edge. If $x=c_{u,L_u}$ is terminal, then by construction $x\in B_{g_O}$ for some cycle $O$, and $m(g_O,\tau(g_O))=1$. Hence
\[
m^-(g_O,\tau(g_O))=m(g_O,\tau(g_O))-1=0,
\]
so there is no core edge from $x$ to $r_{\tau(g_O)}$. Therefore no parallel edge occurs here either.
Thus $\widetilde G$ has no parallel edges. By Lemma~\ref{lem:vertex-rigidity-refinement-permutation}, every automorphism of $\widetilde G$ fixes every vertex. Since there is at most one edge between any ordered pair of vertices, every automorphism fixes every edge as well. Hence $\Aut(\widetilde G)=\{\id\}$.
\end{proof}

The following lemma will be useful later in the proof of the main theorem. 
\begin{lemma}\label{lem:triangular-walks-lift}
Under the assumptions of
Proposition~\ref{prop:vertex-rigid-refinement-permutation},
assume moreover that there exist a vertex $v_0\in V$ and closed walks
$D_1,\dots,D_r$ in $G$, all based at $v_0$, such that, for each $i$,
the edge $f_i$ occurs exactly once in $D_i$ and does not occur in any
of the walks $D_1,\dots,D_{i-1}$. Then the integers $L_u$ in the construction of
Proposition~\ref{prop:vertex-rigid-refinement-permutation} may be chosen
so that the resulting refinement $\widetilde G$ contains closed walks
$\widetilde D_1,\dots,\widetilde D_r$, all based at $r_{v_0}$, such
that, for each $i$, the edge $\widetilde f_i$ occurs exactly once in
$\widetilde D_i$ and does not occur in any of the walks
$\widetilde D_1,\dots,\widetilde D_{i-1}$.
\end{lemma}

\begin{proof}
We choose the refinement in
Proposition~\ref{prop:vertex-rigid-refinement-permutation}
as follows. For each $u\in V$, let $o(u)$ be the length of the
$\tau$-cycle containing $u$, and choose pairwise distinct positive
integers $L_u$ such that
\[
o(u)\mid L_u
\qquad (u\in V).
\]
Such a choice is possible by taking sufficiently large distinct
multiples of the corresponding cycle lengths. Then
\[
\tau^{L_u}(u)=u,
\qquad
\tau^{L_u+1}(u)=\tau(u).
\]
Hence the coding walk associated with $u$ is a walk from $r_u$ to
$r_{\tau(u)}$:
\[
r_u\longrightarrow c_{u,1}\longrightarrow\cdots
\longrightarrow c_{u,L_u}\longrightarrow r_{\tau(u)}.
\]
Let $\widetilde G$ be the refinement obtained from this choice. For each $u,w\in V$, assign to every edge $e\in\delta_u(w)$ a walk
$\widehat e$ in $\widetilde G$ from $r_u$ to $r_w$ as follows. If $w\neq\tau(u)$, there are exactly
\[
m(u,w)=|\delta_u(w)|
\]
core edges from $r_u$ to $r_w$; choose a bijection between
$\delta_u(w)$ and these core edges.

If $w=\tau(u)$, there are
\[
m(u,\tau(u))-1
\]
core edges from $r_u$ to $r_{\tau(u)}$, together with the coding walk
\[
r_u\longrightarrow c_{u,1}\longrightarrow\cdots
\longrightarrow c_{u,L_u}\longrightarrow r_{\tau(u)}.
\]
Since
\[
m(u,\tau(u))
=
\bigl(m(u,\tau(u))-1\bigr)+1,
\]
choose a bijection between $\delta_u(\tau(u))$ and the set consisting
of these core edges and this coding walk.
Thus every edge $e\in E(G)$ is assigned a unique walk $\widehat e$ in
$\widetilde G$ from $r_{s(e)}$ to $r_{t(e)}$. For each $i$, let
$\widetilde D_i$ be obtained from $D_i$ by replacing every edge $e$
with the corresponding walk $\widehat e$. Since $D_i$ is closed and
based at $v_0$, the walk $\widetilde D_i$ is closed and based at
$r_{v_0}$.
Fix $i$. By hypothesis, the edge $f_i$ occurs exactly once in $D_i$
and does not occur in $D_1,\dots,D_{i-1}$. If $\widehat f_i$ is a core edge, set
\[
\widetilde f_i:=\widehat f_i.
\]
Since $f_i$ occurs exactly once in $D_i$, the edge
$\widetilde f_i$ occurs exactly once in $\widetilde D_i$.
By injectivity of the assignment $e\mapsto\widehat e$, it does not
occur in any of $\widetilde D_1,\dots,\widetilde D_{i-1}$.
If $\widehat f_i$ is the coding walk associated with some $u\in V$, set
\[
\widetilde f_i:=
\bigl(r_u\longrightarrow c_{u,1}\bigr).
\]
This edge belongs to no assigned walk other than $\widehat f_i$.
Since $f_i$ occurs exactly once in $D_i$, the edge
$\widetilde f_i$ occurs exactly once in $\widetilde D_i$.
Since $f_i$ does not occur in any of $D_1,\dots,D_{i-1}$,
the edge $\widetilde f_i$ does not occur in any of
$\widetilde D_1,\dots,\widetilde D_{i-1}$. Hence, for each $i$, the edge $\widetilde f_i$ occurs exactly once in $\widetilde D_i$ and does not occur in any of $\widetilde D_1,\dots,\widetilde D_{i-1}$.
\end{proof}
Note that a naive lifting argument is not sufficient in the previous proof: the lift of a closed walk in $G$ from a prescribed initial vertex of $\widetilde G$ need not be closed. Although one could then close it up using strong connectivity of $\widetilde G$, this would introduce additional edges and destroy the control needed for the ``triangularity condition" of the previous lemma. 
Regarding Proposition~\ref{prop:vertex-rigid-refinement-permutation}, in general, one cannot expect $\Aut(\widetilde G)=\{\id\}$, because parallel edges with the same source and target can always be permuted. It would nevertheless be interesting to determine whether one can avoid this phenomenon at the cost of a more complicated construction. A basic situation in which this obstruction disappears is when the initial $k$-out graph has no parallel edges, i.e.\ is parallel-free. In that case, the refinement constructed above is fully rigid.

We now turn to the standard voltage construction. Let $\widetilde G$ be a finite graph, let $H$ be a finite group, and let $\alpha:E(\widetilde G)\to H$ be a voltage assignment. The associated derived graph $\widetilde G^\alpha$ has vertex set $V(\widetilde G)\times H$ and edge set $E(\widetilde G)\times H$, with
\[
s(e,h)=(s(e),h),\qquad t(e,h)=(t(e),h\alpha(e)).
\]
There is a natural projection
\[
\pi:\widetilde G^\alpha\to \widetilde G,
\]
given on vertices and edges by $\pi_v(v,h)=v$ and $\pi_e(e,h)=e$.
The following lemma is specific to the refinement constructed above. In general, automorphisms of a regular or voltage cover must not preserve the fibers of the covering projection a priori; here, this follows from the intrinsic $R/X$-structure of the refinement.

\begin{lemma}\label{lem:fibers-preserved-derived}
Let $\widetilde G$ be the graph constructed in Proposition~\ref{prop:vertex-rigid-refinement-permutation}, let $H$ be a finite group, and let $\alpha:E(\widetilde G)\to H$ be any voltage assignment. Let $G^\alpha:=\widetilde G^\alpha$ be the derived graph, and let $\pi:G^\alpha\to \widetilde G$ be the natural projection. Then every automorphism of $G^\alpha$ preserves each fiber
\[
\pi^{-1}(x)=\{x\}\times H
\qquad (x\in V(\widetilde G))
\]
setwise.
\end{lemma}

\begin{proof}
Write
\[
R:=\{r_u: u \text{ is a root vertex of }\widetilde G\},
\]
and let
\[
X:=V(\widetilde G)\setminus R.
\]
Set similarly
\[
R^\alpha:=R\times H\subseteq V(G^\alpha),
\qquad
X^\alpha:=V(G^\alpha)\setminus R^\alpha.
\]
We first note that in-degrees are preserved by the derived construction. Indeed, fix $(x,g)\in V(G^\alpha)$. For every edge $e:y\to x$ of $\widetilde G$, there is a unique edge $(e,h):(y,h)\to (x,g)$ of $G^\alpha$, corresponding to $h=g\alpha(e)^{-1}$. Hence $|\delta_{V(G^\alpha)}((x,g))|=|\delta_{V(\widetilde G)}(x)|.$ By the proof of Lemma~\ref{lem:vertex-rigidity-refinement-permutation}, in $\widetilde G$ one has
\[
R=\{x\in V(\widetilde G): |\delta_{V(\widetilde G)}(x)|\ge 2\}.
\]
Therefore
\[
R^\alpha=\{(x,h)\in V(G^\alpha): x\in R\},
\]
so every automorphism of $G^\alpha$ preserves $R^\alpha$ setwise, and hence also $X^\alpha$. Now fix $(r_u,h)\in R^\alpha$. The distinguished walk
$r_u\to c_{u,1}\to \cdots \to c_{u,L_u}$ in $\widetilde G$ lifts uniquely to a walk $(r_u,h)=z_0\to z_1\to \cdots \to z_{L_u}$ in $G^\alpha$, with $z_1,\dots,z_{L_u}\in X^\alpha$. Since the vertices $r_u,c_{u,1},\dots,c_{u,L_u}$
are pairwise distinct in $\widetilde G$, the lifted vertices $z_0,z_1,\dots,z_{L_u}$ are pairwise distinct as well. Thus this lifted walk is in fact a directed path.
Moreover, $(r_u,h)$ has exactly one outgoing edge to $X^\alpha$, each $z_i$ with $1\le i<L_u$ has exactly one outgoing edge to $X^\alpha$, and $z_{L_u}$ has none. Hence this is the unique maximal directed path starting at $(r_u,h)$ whose vertices after the first lie in $X^\alpha$. In particular, its length is exactly $L_u$. Since the integers $L_u$ are pairwise distinct, the set
\[
F_u:=\{(r_u,h): h\in H\}
\]
is characterized intrinsically as the set of vertices of $R^\alpha$ from which the unique maximal path through $X^\alpha$ has length $L_u$. Hence every automorphism of $G^\alpha$ preserves each $F_u$ setwise. Finally, for each $1\le i\le L_u$, the set
\[
F_{u,i}:=\{(c_{u,i},h): h\in H\}
\]
is exactly the set of vertices reached after $i$ steps along the unique maximal path through $X^\alpha$ starting from vertices of $F_u$. Hence every automorphism also preserves each $F_{u,i}$ setwise. Since every vertex of $\widetilde G$ is either some $r_u$ or some $c_{u,i}$, every fiber $\{x\}\times H$ is preserved setwise.
\end{proof}
We shall use the following standard sufficient criterion for strong connectivity in the setting of voltage and derived graphs. We recall that for a walk $W=e_1e_2\cdots e_m$ in a voltage graph, we write
\[
\alpha(W):=\alpha(e_1)\alpha(e_2)\cdots\alpha(e_m)\in H
\]
for its voltage.
\begin{proposition}\label{prop:strong-connectivity-derived}
Let $\Delta$ be a finite strongly connected directed graph, let $H$ be a finite group, and let $\alpha:E(\Delta)\to H$ be a voltage assignment. Fix a vertex $v_0\in V(\Delta)$, and let
\[
K_{v_0}:=\langle \alpha(W): W \text{ is a closed walk in }\Delta\text{ based at }v_0\rangle\leq H.
\]
If $K_{v_0}=H$, then the derived graph $\Delta^\alpha$ is strongly connected.
\end{proposition}
\begin{proof}
Let $(u,g),(v,h)\in V(\Delta^\alpha)$. Since $\Delta$ is strongly connected, choose a path $P$ in $\Delta$ from $u$ to $v$, and write $\alpha(P)=a\in H$. Then $P$ lifts to a  path in $\Delta^\alpha$ from $(u,g)$ to $(v,ga)$. Choose directed paths $Q:v\to v_0$ and $R:v_0\to v$ in $\Delta$. Since $K_{v_0}=H$, there exists a closed walk $W$ at $v_0$ such that
\[
\alpha(W)=\alpha(Q)^{-1}a^{-1}g^{-1}h\,\alpha(R)^{-1}.
\]
Hence the closed walk $QWR$ based at $v$ has voltage $a^{-1}g^{-1}h$, and its lift gives a directed path from $(v,ga)$ to $(v,h)$. Therefore $(u,g)$ reaches $(v,h)$, so $\Delta^\alpha$ is strongly connected.
\end{proof}
\begin{remark}\label{rem: combinatorial condition}
The hypothesis $K_{v_0}=H$ is sufficient. In particular, strong connectivity of the base graph alone does not force strong connectivity of the derived graph. A convenient sufficient condition for arranging $K_{v_0}=H$ is the following. Suppose that there exist closed walks $C_1,\dots,C_r$ in $\Delta$, all based at $v_0$, such that, for each $i$, the edge $e_i$ occurs exactly once in $C_i$
and does not occur in any of the walks $C_1,\dots,C_{i-1}$. Then, for any choice of generators $h_1,\dots,h_r$ of $H$, one can choose a voltage assignment $\alpha:E(\Delta)\to H$ such that $\alpha(C_i)=h_i$ for $1\le i\le r$. In particular, for this choice of $\alpha$ one has $K_{v_0}=H$. Indeed, assign voltages inductively. At stage $i$, the value of $\alpha(e_i)$ has not yet been used in any of
$C_1,\dots,C_{i-1}$. Since $e_i$ occurs exactly once in $C_i$, one may
write
\[
\alpha(C_i)=a_i\,\alpha(e_i)\,b_i,
\]
where $a_i,b_i\in H$ are already determined. Setting
\[
\alpha(e_i):=a_i^{-1}h_i b_i^{-1}
\]
gives $\alpha(C_i)=h_i$ without changing any of the previously
prescribed voltages.
The above condition may be viewed as a combinatorial strengthening of the requirement that the fundamental group of the underlying $1$-complex have rank at least the rank $d(H)$ of the group $H$ (the minimal size of a generating set of $H$).
\end{remark}
The next argument is independent of the connectivity of $\Gamma^\alpha$.
\begin{proposition}\label{prop:derived-k-out-not-ts}
Let $\Gamma=(V_\Gamma,E_\Gamma,s,t)$ be a $k$-out graph, let $H$ be a nontrivial finite group, and let $\alpha:E_\Gamma\to H$ be a voltage assignment. Then the derived graph $\Gamma^\alpha$ is a $k$-out graph which is not totally synchronizing.
\end{proposition}

\begin{proof}
For every $(v,h)\in V_\Gamma\times H$, the outgoing edges of $(v,h)$ are exactly the pairs
$(e,h)$ with $s(e)=v$. Hence $\Gamma^\alpha$ is $k$-out. Choose $\sigma\in H\setminus\{1\}$, and define an automorphism $\lambda_\sigma$ of $\Gamma^\alpha$ by
$\lambda_\sigma(v,h)=(v,\sigma h)$ and $\lambda_\sigma(e,h)=(e,\sigma h)$.
Indeed,
$s(\lambda_\sigma(e,h))=(s(e),\sigma h)=\lambda_\sigma(s(e,h))$ and
$t(\lambda_\sigma(e,h))=(t(e),\sigma h\alpha(e))=\lambda_\sigma(t(e,h))$.
Moreover, $\lambda_\sigma$ has no fixed vertex. The automorphism $\lambda_\sigma$ is vertex-rigid. Indeed, if $\lambda_\sigma^m$ fixes some vertex
$(v,h)$, then $(v,\sigma^m h)=(v,h)$, hence $\sigma^m=1$, and therefore $\lambda_\sigma^m=\id$. By Proposition~\ref{prop: rigid in end of automaton}, there exists a coloring $\chi$ of
$\Gamma^\alpha$ such that $\lambda_\sigma\in\End(\chi(\Gamma^\alpha))$. Since $\lambda_\sigma$ is
bijective, it is an automorphism of the automaton $\chi(\Gamma^\alpha)$.
Since $\lambda_\sigma$ has no fixed vertex, Proposition~\ref{pro: synch and automorphisms} implies
that $\chi(\Gamma^\alpha)$ is not synchronizing. Therefore $\Gamma^\alpha$ is not totally
synchronizing.
\end{proof}

\begin{lemma}\label{lem:frucht_extension_mod_edge}
Let $\Gamma=(V_\Gamma,E_\Gamma,s,t)$ be a finite strongly connected $k$-out graph, and let $H$ be a nontrivial finite group. Let $r=d(H)$ be the minimal size of a generating set of $H$. Suppose that:
\begin{itemize}
    \item there exists a permutation $\tau:V_\Gamma\to V_\Gamma$ such that $|\delta_u(\tau(u))|\ge 1$ for every $u\in V_\Gamma$, and such that the graph
    \[
    \Gamma^-=(V_\Gamma,E_\Gamma^-,s,t),
    \]
    obtained by removing exactly one edge from $\delta_u(\tau(u))$ for each $u\in V_\Gamma$, is still strongly connected;

    \item there exist a vertex $v_0\in V_\Gamma$ and closed walks $D_1,\dots,D_r$ in $\Gamma$, all based at $v_0$, such that, for each $i$, the edge $f_i$ occurs exactly once in $D_i$ and does not occur in any of the walks $D_1,\dots,D_{i-1}$.
\end{itemize}
Assume moreover that the voltage assignment $\alpha:E(\widetilde\Gamma)\to H$ produced from this data can be chosen so that parallel edges of $\widetilde\Gamma$ receive the same voltage. Then there exist a finite strongly connected $k$-out graph $G$ and a graph congruence $\rho$ on $G$ such that
\[
G/\rho\cong \Gamma,\qquad \Aut(G)/\Aut_E(G)\cong H,
\]
and $G$ is not totally synchronizing.
\end{lemma}

\begin{proof}
Choose the refinement in
Proposition~\ref{prop:vertex-rigid-refinement-permutation}
with the integers $L_u$ prescribed in
Lemma~\ref{lem:triangular-walks-lift}.
Thus there exist a finite strongly connected $k$-out graph
$\widetilde\Gamma$ and a graph congruence $\rho_0$ on
$\widetilde\Gamma$ such that
\[
\widetilde\Gamma/\rho_0\cong\Gamma,
\qquad
\Aut(\widetilde\Gamma)=\Aut_E(\widetilde\Gamma),
\]
and there exist closed walks
$\widetilde D_1,\dots,\widetilde D_r$ in $\widetilde\Gamma$, all based
at $r_{v_0}$, such that, for each $i$, the edge $\widetilde f_i$ occurs exactly once in $\widetilde D_i$ and does not occur in any of
$\widetilde D_1,\dots,\widetilde D_{i-1}$.
Choose generators $h_1,\dots,h_r$ of $H$. By remark~\ref{rem: combinatorial condition}, there exists a voltage assignment $\alpha:E(\widetilde\Gamma)\to H$ such that
\[
K_{r_{v_0}}:=\langle \alpha(W): W \text{ is a closed walk in }\widetilde\Gamma\text{ based at }r_{v_0}\rangle = H,
\]
and, by assumption, this assignment may be chosen so that parallel edges of $\widetilde\Gamma$ receive the same voltage. Let $G=\widetilde\Gamma^\alpha$. Since $\widetilde\Gamma$ is strongly connected and $K_{r_{v_0}}=H$, Proposition~\ref{prop:strong-connectivity-derived} implies that $G$ is strongly connected. By Proposition~\ref{prop:derived-k-out-not-ts}, the graph $G$ is $k$-out and is not totally synchronizing.

Let $\pi=(\pi_v,\pi_e):G\to \widetilde\Gamma$ be the natural projection, and define an equivalence relation $\rho$ on $V(G)$ by $(v,h)\mathrel{\rho}(v',h')$ if and only if $v\mathrel{\rho_0}v'$. Let $C$ be a $\rho_0$-class in $\widetilde\Gamma$, and let $\widehat C:=\{(w,g)\in V(G): w\in C\}$ be the corresponding $\rho$-class in $G$. For every $(v,h)\in V(G)$ one has $|\delta_{(v,h)}(\widehat C)|=|\delta_v(C)|$, because the outgoing edges from $(v,h)$ are exactly the pairs $(e,h)$ with $s(e)=v$, and $t(e,h)=(t(e),h\alpha(e))$ belongs to $\widehat C$ if and only if $t(e)\in C$. Since $\rho_0$ is a congruence on $\widetilde\Gamma$, the quantity $|\delta_v(C)|$ depends only on the $\rho_0$-class of $v$. Hence $\rho$ is a graph congruence, and $G/\rho\cong \widetilde\Gamma/\rho_0\cong \Gamma$.

For each $\sigma\in H$, define $\lambda_\sigma=(\lambda_{\sigma,v},\lambda_{\sigma,e})$ by
\[
\lambda_{\sigma,v}(v,h):=(v,\sigma h),\qquad
\lambda_{\sigma,e}(e,h):=(e,\sigma h).
\]
Then $\lambda_\sigma$ is an automorphism of $G$, so $\sigma\mapsto \lambda_\sigma$ gives an injective homomorphism $H\hookrightarrow \Aut(G)$. Moreover, by Lemma~\ref{lem:fibers-preserved-derived}, every automorphism of $G$ preserves each vertex-fiber $\{v\}\times H$ setwise. In particular, each $\lambda_\sigma$ preserves every such fiber. If $\lambda_\sigma\in \Aut_E(G)$, then it fixes every vertex of $G$, hence
\[
(v,\sigma h)=\lambda_{\sigma,v}(v,h)=(v,h)
\qquad\text{for all }(v,h)\in V(G).
\]
This forces $\sigma=1$. Therefore the image of $H$ in $\Aut(G)$ meets $\Aut_E(G)$ trivially, and so the map $\sigma\mapsto \lambda_\sigma$ induces an injective homomorphism
\[
H\hookrightarrow \Aut(G)/\Aut_E(G).
\]
Conversely, let $\psi\in\Aut(G)$. By Lemma~\ref{lem:fibers-preserved-derived}, $\psi$ preserves each vertex-fiber $\{v\}\times H$ setwise. Hence there exist permutations $\theta_v\in\Sym(H)$ such that
\[
\psi(v,h)=(v,\theta_v(h))
\qquad ((v,h)\in V(G)).
\]
Let $e:u\to w$ be an edge of $\widetilde\Gamma$. Since $(u,h)\xrightarrow{(e,h)}(w,h\alpha(e))$ is an edge of $G$, its image under $\psi$ is an edge from $(u,\theta_u(h))$ to $(w,\theta_w(h\alpha(e)))$. Hence $\psi(e,h)=(e',\theta_u(h))$ for some edge $e':u\to w$ in $\widetilde\Gamma$, and therefore
\[
\theta_w(h\alpha(e))=\theta_u(h)\alpha(e').
\]
Since $e'$ is parallel to $e$ and parallel edges have the same voltage by assumption, we get $\alpha(e')=\alpha(e)$. Therefore
\[
\theta_w(h\alpha(e))=\theta_u(h)\alpha(e)
\qquad(h\in H).
\]
By induction on the length, the same formula holds for every walk. In particular, if $W$ is a closed walk based at $r_{v_0}$, then $\theta_{r_{v_0}}(h\,\alpha(W))=\theta_{r_{v_0}}(h)\,\alpha(W)$. Since $K_{r_{v_0}}=H$, it follows that $\theta_{r_{v_0}}(hx)=\theta_{r_{v_0}}(h)x$ for all $h,x\in H$. Taking $h=1$, we obtain $\theta_{r_{v_0}}(x)=\theta_{r_{v_0}}(1)x$, so $\theta_{r_{v_0}}$ is left translation by the element $\sigma:=\theta_{r_{v_0}}(1)\in H$. Now let $u\in V(\widetilde\Gamma)$ and choose a path $P_u$ from $u$ to $r_{v_0}$, with voltage $\alpha(P_u)=a_u$. Applying the walk relation to $P_u$, we obtain $\theta_{r_{v_0}}(h a_u)=\theta_u(h)a_u$, hence we deduce $\theta_u(h)=\theta_{r_{v_0}}(h a_u)a_u^{-1}$. Since $\theta_{r_{v_0}}$ is left translation by $\sigma$, it follows that $\theta_u(h)=(\sigma h a_u)a_u^{-1}=\sigma h$. Therefore $\psi_v(u,h)=(u,\sigma h)$ for every $(u,h)\in V(G)$. Therefore $\lambda_\sigma^{-1}\psi$ fixes every vertex of $G$, i.e. belongs to $\Aut_E(G)$. This shows that every class in $\Aut(G)/\Aut_E(G)$ is represented by some deck transformation $\lambda_\sigma$, and hence the natural map $H\to \Aut(G)/\Aut_E(G)$ is surjective. Consequently, $\Aut(G)/\Aut_E(G)\cong H$.
\end{proof}
We are now in position to prove the main theorem of this section.
\begin{theorem}\label{thm:one-loop-per-vertex}
Let $\Gamma=(V_\Gamma,E_\Gamma,s,t)$ be a finite strongly connected $k$-out graph, and let $H$ be a nontrivial finite group. Assume that $|V_\Gamma|\ge d(H)$. Let $\Gamma^\sharp$ be the graph obtained from $\Gamma$ by adding exactly one loop at every vertex. Then $\Gamma^\sharp$ is a strongly connected $(k+1)$-out graph, and there exist a finite strongly connected $(k+1)$-out graph $G$ and a graph congruence $\rho$ on $G$ such that
\[
G/\rho\cong \Gamma^\sharp,\qquad \Aut(G)/\Aut_E(G)\cong H,
\]
and $G$ is not totally synchronizing.
\end{theorem}
\begin{proof}
Set $r:=d(H)$. Since $|V_\Gamma|\ge r$, choose pairwise distinct vertices $u_1,\dots,u_r\in V_\Gamma$, and fix a base vertex $v_0\in V_\Gamma$. For each $i$, let $\ell_i$ denote the loop added at $u_i$. Since $\Gamma$ is strongly connected, for each $i$ there exist directed paths $P_i$ from $v_0$ to $u_i$ and $Q_i$ from $u_i$ to $v_0$. Choosing these paths without repeated vertices, we may assume that none of them uses a loop. Define a closed walk based at $v_0$ by $D_i:=P_i\,\ell_i\,Q_i$. Then, for each $i$, the edge $\ell_i$ occurs exactly once in $D_i$ and does not occur in any of the walks $D_1,\dots,D_{i-1}$. Indeed, the loops $\ell_1,\dots,\ell_r$ are pairwise distinct, and none of the
paths $P_j,Q_j$ uses a loop. Now apply Proposition~\ref{prop:vertex-rigid-refinement-permutation} to the graph $\Gamma^\sharp$, taking $\tau=\id_{V_\Gamma}$. The first hypothesis is automatic, since every vertex of $\Gamma^\sharp$ carries a loop, and removing those added loops recovers $\Gamma$, which is still strongly connected. Thus we obtain a finite strongly connected $(k+1)$-out graph $\widetilde\Gamma$ and a graph congruence $\rho_0$ on $\widetilde\Gamma$ such that $\widetilde\Gamma/\rho_0\cong \Gamma^\sharp$ and $\Aut(\widetilde\Gamma)=\Aut_E(\widetilde\Gamma)$. By Lemma~\ref{lem:triangular-walks-lift}, the walks
$D_1,\dots,D_r$ lift to closed walks
$\widetilde D_1,\dots,\widetilde D_r$ in $\widetilde\Gamma$, all based
at $r_{v_0}$, such that, for each $i$, the edge $\widetilde f_i$
occurs exactly once in $\widetilde D_i$ and does not occur in any of
the walks $\widetilde D_1,\dots,\widetilde D_{i-1}$. Choose generators $h_1,\dots,h_r$ of $H$. By remark~\ref{rem: combinatorial condition} one can choose a voltage assignment $\alpha:E(\widetilde\Gamma)\to H$ such that $K_{r_{v_0}}=H$. Moreover, for each $i$, the edge $f_i$ is the added loop at the
vertex $u_i$, and in Lemma~\ref{lem:triangular-walks-lift} the
corresponding distinguished edge may be chosen as
\[
\widetilde f_i=r_{u_i}\longrightarrow c_{u_i,1}.
\]
Each $\widetilde f_i$ is the unique edge of $\widetilde\Gamma$ ending
at $c_{u_i,1}$, and therefore it has no parallel edge. Consequently,
in the triangular construction of the voltage assignment, the value
of $\alpha(\widetilde f_i)$ may be chosen independently, while every
non-distinguished parallel class is assigned a common voltage.
Thus $\alpha$ may be chosen so that
\[
K_{r_{v_0}}=H
\]
and parallel edges of $\widetilde\Gamma$ receive the same voltage. Now apply Lemma~\ref{lem:frucht_extension_mod_edge} to $\Gamma^\sharp$ to conclude the proof. 
\end{proof}
An interesting case in which the vertex-trivial automorphisms disappear is given by the following corollary. Indeed, if the initial graph has no parallel edges and no loops, then after adding one loop at each vertex the graph $\Gamma^\sharp$ is still parallel-free, and one checks directly from the refinement and derived-graph constructions that no parallel edges are created at any subsequent step. We leave the straightforward verification to the reader. In particular, the final graph $G$ is parallel-free, hence $\Aut_E(G)=1$.
\begin{corollary}\label{cor:one-loop-per-vertex-parallel-free}
Let $\Gamma=(V_\Gamma,E_\Gamma,s,t)$ be a finite strongly connected $k$-out graph without parallel edges and without loops, and let $H$ be a nontrivial finite group. Assume that $|V_\Gamma|\ge d(H)$. Let $\Gamma^\sharp$ be the graph obtained from $\Gamma$ by adding exactly one loop at every vertex. Then there exist a finite strongly connected $(k+1)$-out graph $G$ and a graph congruence $\rho$ on $G$ such that
\[
G/\rho\cong \Gamma^\sharp,\qquad \Aut(G)\cong H,
\]
and $G$ is not totally synchronizing.
\end{corollary}
The previous results provide a broad family of $k$-out graphs that are not totally synchronizing and admit many symmetries, even in situations where the corresponding quotient may itself be totally synchronizing. This naturally raises the problem of constructing totally synchronizing $k$-out graphs with a prescribed automorphism group. The voltage-lift approach developed above cannot be used for this purpose. Indeed, every nontrivial derived graph carries nontrivial fixed-point-free deck transformations. These transformations are vertex-rigid, hence they become automorphisms of suitable colorings; by Proposition~\ref{pro: synch and automorphisms}, every such derived graph admits a non-synchronizing coloring. Thus the realization problem in the totally synchronizing setting appears to be genuinely different from the non-synchronizing one. This motivates the following question.
\begin{open}
Given a finite group $H$ satisfying the conditions of Theorem~\ref{theo: main}, characterize those strongly connected $k$-out graphs $G$ which are totally synchronizing and satisfy $\Aut(G)/\Aut_E(G)\cong H$. More specifically, determine which finite groups can occur as automorphism groups of totally synchronizing $k$-out graphs.
\end{open}
Nonetheless, we give a partial result by exploiting the ``rigid refinement" construction of Proposition~\ref{prop:vertex-rigid-refinement-permutation}.

\begin{lemma}\label{lem:id-refinement-pf-criterion}
Assume that in the refinement construction of Proposition~\ref{prop:vertex-rigid-refinement-permutation} one takes $\tau=\id$. Let $\widetilde\Gamma$ be the resulting rigid refinement, let $\mu\in \mathbb N_{>0}^{V(\Gamma)}$ be the integer-normalized left Perron--Frobenius eigenvector of the quotient graph $\Gamma$, and let $\pi\in\mathbb N_{>0}^{V(\widetilde\Gamma)}$ be the integer-normalized left Perron--Frobenius eigenvector of $\widetilde\Gamma$. For each $u\in V(\Gamma)$, set $\beta_u:=1+\frac1k+\cdots+\frac1{k^{L_u}}$. If there exists $u\in V(\Gamma)$ such that
\[
\frac{\mu(u)}{\beta_u}>\frac12\sum_{v\in V(\Gamma)}\mu(v),
\]
then $\pi$ is not equipartitionable. Hence $\widetilde\Gamma$ is totally synchronizing.
\end{lemma}
\begin{proof}
Since $\tau=\id$, each block has the form $B_u=\{r_u,c_{u,1},\dots,c_{u,L_u}\}$, and the coding chain inside $B_u$ is $r_u\to c_{u,1}\to \cdots \to c_{u,L_u}\to r_u$. Fix $u\in V(\Gamma)$. For each $1\le i\le L_u$, the only edge entering $c_{u,i}$ is the coding edge from its predecessor in the chain. Therefore $k\,\pi(c_{u,1})=\pi(r_u)$ and, for $2\le i\le L_u$, one has $k\,\pi(c_{u,i})=\pi(c_{u,i-1})$. By induction, $\pi(c_{u,i})=\pi(r_u)/k^i$ for all $1\le i\le L_u$. Hence
\(
\sum_{x\in B_u}\pi(x)=\pi(r_u)\beta_u.
\)
Now apply Lemma~\ref{lem:pf-lump-part} to the lumpable partition $\rho_0$ of $\widetilde\Gamma$. Since $\widetilde\Gamma/\rho_0\simeq \Gamma$, there exists $\alpha\in\mathbb N_{>0}$ such that $\sum_{x\in B_u}\pi(x)=\alpha\,\mu(u)$ for every $u\in V(\Gamma)$. Therefore $\pi(r_u)\beta_u=\alpha\,\mu(u)$, hence $\pi(r_u)=\alpha\,\mu(u)/\beta_u$.
Now suppose that there exists $u\in V(\Gamma)$ such that $\mu(u)/\beta_u>\frac12\sum_{v\in V(\Gamma)}\mu(v)$. Multiplying by $\alpha$, we get $\pi(r_u)>\frac{\alpha}{2}\sum_{v\in V(\Gamma)}\mu(v)$. On the other hand, $\sum_{x\in V(\widetilde\Gamma)}\pi(x)=\sum_{v\in V(\Gamma)}\sum_{x\in B_v}\pi(x)=\alpha\sum_{v\in V(\Gamma)}\mu(v)$. Thus
\(
\pi(r_u)>\frac12\sum_{x\in V(\widetilde\Gamma)}\pi(x).
\)
By Lemma~\ref{lem:pf-dominating-coordinate}, $\pi$ is not equipartitionable, and thus $\widetilde\Gamma$ is totally synchronizing.
\end{proof}

\begin{corollary}\label{cor:bouquet-rigidification-ts}
Let $\Gamma$ be the bouquet with one vertex and $k$ loops, and let $\widetilde\Gamma$ be the rigid refinement obtained by choosing a single coding chain of length $L\ge 1$. Then $\widetilde\Gamma$ is totally synchronizing.
\end{corollary}
\begin{proof}
In this case $\Gamma$ has a single vertex $u$, its integer-normalized left Perron--Frobenius eigenvector is $\mu(u)=1$, and $\beta_u=1+\frac1k+\cdots+\frac1{k^L}<2$. Hence
\(
\frac{\mu(u)}{\beta_u}=\frac1{\beta_u}>\frac12=\frac12\sum_{v\in V(\Gamma)}\mu(v).
\)
Now apply Lemma~\ref{lem:id-refinement-pf-criterion}.
\end{proof}

There is another simple sufficient condition for total synchronization in the case $\tau=\id$.
\begin{corollary}
Let $\mu$ be the integer-normalized left Perron--Frobenius eigenvector of the quotient $\Gamma$.
If some $u\in V(\Gamma)$ satisfies
\(
\mu(u)>\frac{k}{2(k-1)}\sum_{v\in V(\Gamma)}\mu(v),
\)
then $\widetilde\Gamma$ is totally synchronizing, 
\end{corollary}
\begin{proof}
Since $\beta_u<\frac{k}{k-1}$ implies
$\mu(u)/\beta_u>\frac{k-1}{k}\mu(u)>\frac12\sum_v\mu(v)$. then Lemma~\ref{lem:id-refinement-pf-criterion} applies.
\end{proof}

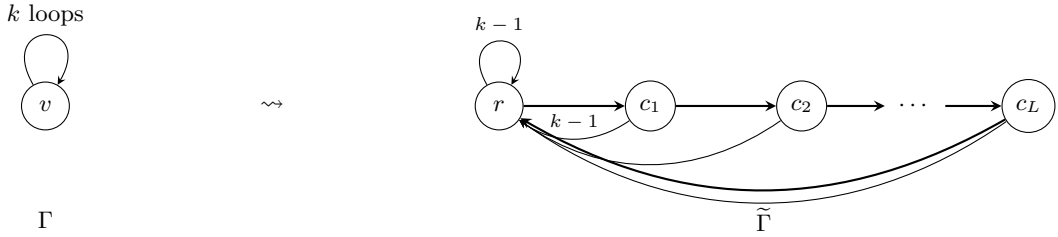
\begin{figure}[ht]
\centering
\begin{tikzpicture}[>=stealth, every node/.style={font=\small}]
  \begin{scope}[xshift=0cm]
    \node[circle,draw,minimum size=18pt] (v) at (0,0) {$v$};

    \draw[->,looseness=8,out=120,in=60] (v) to node[above] {$k$ loops} (v);

    \node at (0,-1.5) {$\Gamma$};
  \end{scope}

  \node at (3,0) {$\rightsquigarrow$};

  \begin{scope}[xshift=6cm]
    \node[circle,draw,minimum size=18pt] (r) at (0,0) {$r$};
    \node[circle,draw,minimum size=18pt] (c1) at (2,0) {$c_1$};
    \node[circle,draw,minimum size=18pt] (c2) at (4,0) {$c_2$};
    \node at (5.5,0) {$\cdots$};
    \node[circle,draw,minimum size=18pt] (cL) at (7,0) {$c_L$};

    \draw[->,thick] (r) -- (c1);
    \draw[->,thick] (c1) -- (c2);
    \draw[->,thick] (c2) -- (5.1,0);
    \draw[->,thick] (5.9,0) -- (cL);
    \draw[->,thick,bend left=30] (cL) to (r);

    \draw[->,bend left=35] (c1) to node[above] {\scriptsize $k-1$} (r);
    \draw[->,bend left=35] (c2) to (r);
    \draw[->,bend left=35] (cL) to (r);
    \draw[->,looseness=7,out=120,in=60] (r) to node[above] {\scriptsize $k-1$} (r);

    \node at (3.5,-1.5) {$\widetilde{\Gamma}$};
  \end{scope}
\end{tikzpicture}
\caption{The bouquet with one vertex and $k$ loops, and its rigid refinement with one coding chain of length $L$. In the refinement, the distinguished cycle is $r\to c_1\to \cdots \to c_L\to r$, and every vertex carries $k-1$ additional edges to $r$.}
\end{figure}

\section{Totally simple graphs}
It turn out that lumpability is connected to congruences of automata, for this reason we give the following definition. 
\begin{definition}
A $k$-out graph $G$ is called totally simple if for any coloring $\chi$, $\chi(G)$ is a simple automaton. 
\end{definition}
Note that for every non-trivial maximal congruence $\rho\in\Cong(G)$, $G/\rho$ is totally-simple. 
\begin{lemma}\label{lem:lifting-coloring}
Let $G=(V,E,s,t)$ be a finite $k$-out graph and let $\rho_{\mathcal C}$ be a congruence on $G$ with equivalence classes $\mathcal C=\{C_1,\dots,C_m\}$. Let $\varphi:G\to G/\rho_{\mathcal{C}}$ denote the quotient map.
Every valid coloring $\chi$ of the quotient $G/\rho_{\mathcal C}$ lifts to a valid coloring $\widetilde{\chi}$ of $G$ such that $\rho_{\mathcal C}\in \mathrm{Cong}(\widetilde{\chi}(G))$; moreover, $\widetilde{\chi}$ induces $\chi$ on the quotient.
\end{lemma}
\begin{proof}
Fix a coloring $\chi$ of $G/\rho_{\mathcal C}$ with alphabet $[k]$, and write $\delta'$ for the transition function of the
colored quotient automaton $\chi(G/\rho_{\mathcal C})$. For each class $C_i\in\mathcal C$, the coloring $\chi$ labels the $k$ outgoing edges of the quotient state $C_i$ bijectively by $[k]$. For every $s\in [k]$, let $C_j=\delta'(C_i,s)$, since $\rho_{\mathcal C}$ is a congruence, the underlying quotient $G/\rho_{\mathcal C}$ has the property that for each pair $(i,j)$ the number of edges from any vertex $u\in C_i$ to the block $C_j$ is constant; denote it by $k_{ij}$, so that $\sum_{j=1}^m k_{ij}=k$. Now, for each $i$ and $j$, define $L_{i,j}\ :=\ \{t\in[k]:\ \delta'(C_i,t)=C_j\}$. By construction, the sets $\{L_{i,1},\dots,L_{i,m}\}$ form a partition of $[k]$, and $|L_{i,j}|=k_{ij}$. For each vertex $u\in C_i$, we label its outgoing edges as follows: for every $j$, choose an arbitrary bijection between the set of edges
$\delta_{C_j}(u)=\{e\in E: s(e)=u,\ t(e)\in C_j\}$ (which has size $k_{ij}$) and the label set $L_{i,j}$, and assign labels accordingly.
Doing this for all $u$ yields a valid coloring $\widetilde{\chi}$ of $G$. It remains to check that $\rho_{\mathcal C}$ is a congruence for the automaton $\widetilde{\chi}(G)=(V, [k], \delta)$. Let $u,v\in C_i$ and let $s\in [k]$. By definition, there exists a unique $j$ such that $s\in L_{i,j}$; hence the
$s$-labeled edge leaving $u$ (resp.\ $v$) was chosen among the edges from $u$ to $C_j$ (resp.\ from $v$ to $C_j$). Therefore $\delta(u,s)\in C_j$ and $\delta(v,s)\in C_j$, so $\delta(u,s)\,\rho_{\mathcal C}\,\delta(v,s)$, that is, $\rho_{\mathcal C}\in \mathrm{Cong}(\widetilde{\chi}(G))$, and by construction the induced quotient coloring is $\chi$.
\end{proof}

\begin{lemma}\label{lem: lump=congr}
Let $G=(V,E,s,t)$ be a finite $k$-out directed graph, and let $\rho_{\mathcal C}$ be an equivalence relation on $V$ with classes $\mathcal C=\{C_1,\dots,C_m\}$. Then $\mathcal C$ is lumpable if and only if there exists a coloring $\chi$ such that $\rho_{\mathcal{C}}\in \mathrm{Cong}(\chi(G))$.
\end{lemma}
\begin{proof}
Assume there exists a valid coloring $\chi$ such that $\rho_{\mathcal{C}}$ is a congruence on the resulting automaton $\chi(G)$. Let $u,v\in C_i$. Since $\rho_{\mathcal{C}}$ is a congruence, for every color $s\in [k]$ we have $\delta(u,s)\rho_{\mathcal{C}} \delta(v,s)$. In particular, for any block $C_j$ the set of colors $s$ for which $\delta(u,s)\in C_j$
coincides with the set of colors for which $\delta(v,s)\in C_j$. Hence the number of edges from $u$ to $C_j$ equals the number of edges
from $v$ to $C_j$, and since $i,j$ were arbitrary, $\mathcal C$ is lumpable.
Conversely, assume that $\mathcal C$ is lumpable. Then the quotient multigraph $G/\rho_{\mathcal C}$ is well-defined and $k$-out multigraph.
Choose any valid coloring $\chi$ of the quotient $G/\rho_{\mathcal C}$. By Lemma~\ref{lem:lifting-coloring}, $\chi$ lifts to a valid coloring $\chi$ of $G$ such that $\rho_{\mathcal C}\in\mathrm{Cong}(\chi(G))$. This completes the proof.
\end{proof}
As an immediate consequence of the previous Lemma~\ref{lem: lump=congr} we have. 
\begin{theorem}\label{theo: lump=to simple}
Let $G$ be a $k$-out directed graph. Then $G$ is not lumpable if and only if $G$ is totally simple. 
\end{theorem}

Interestingly, the $2$-out graphs underlying the classical \v{C}ern\'y series (whose colorings are well known to be synchronizing) are also totally simple; Fig.~\ref{fig:cerny-underlying} depicts the case with $6$ vertices.
\begin{proposition}\label{prop:cerny-underlying-totally-simple}
For every $n\ge 2$, the $2$-out graph underlying the \v{C}ern\'y automaton $C_n$ (i.e.\ the graph in which
$0$ has two parallel edges to $1$, every vertex $i\neq 0$ has a self-loop, and every vertex has a successor edge
$i\to i+1$ modulo $n$) is totally simple.
\end{proposition}
\begin{proof}
Let $\rho$ be a lumpable partition (equivalently, a congruence) of $G$.
We show that $\rho$ is either discrete or universal. First, $0$ and $1$ cannot belong to the same $\rho$-class.
Indeed, if $0\,\rho\,1$, then in the quotient $G/\rho$ the class $[0]_\rho=[1]_\rho$ has two loops. Since $G/\rho$ is a $2$-out graph, having two loops at a vertex means that this vertex has no outgoing edges to any other class.
As $G/\rho$ is strongly connected (because $G$ is strongly connected and $\rho$ is lumpable), it follows that $G/\rho$ has only one
vertex, hence $\rho=\nabla$. Next, if $0\,\rho\,i$ for some $i\notin\{0,1\}$, then $[0]_\rho$ contains a vertex with a self-loop, so $G/\rho$ has a loop at $[0]_\rho$.
By lumpability, every vertex in $[0]_\rho$ must have the same number of outgoing edges into $[0]_\rho$; in particular, it must be at least $1$.
But $0$ has no self-loop and both its outgoing edges go to $1$, so the only way for $0$ to have an edge into $[0]_\rho$ is that $1\in[0]_\rho$,
which brings us back to the previous paragraph and forces $\rho=\nabla$.
Therefore, if $\rho$ is non-universal then $[0]_\rho=\{0\}$. Assume now that $\rho$ is non-universal, so $[0]_\rho=\{0\}$.
Since $n-1$ has the successor edge $(n-1)\to 0$, in the quotient there is an edge $[n-1]_\rho\to[0]_\rho=\{0\}$.
Hence every vertex $j\in[n-1]_\rho$ must have exactly one outgoing edge into $\{0\}$. However, the only vertex that has an edge to $0$ is $n-1$ itself; thus $[n-1]_\rho=\{n-1\}$. Repeating the same argument backwards along the cycle, for $i=n-2,n-3,\dots,1$ the quotient has an edge $[i]_\rho\to[i+1]_\rho=\{i+1\}$, so every vertex in $[i]_\rho$ must have exactly one outgoing edge into $\{i+1\}$; but the only vertex with an edge to $i+1$ (different to $i+1$) is $i$. Therefore $[i]_\rho=\{i\}$ for all $i$, and thus $\rho=\Delta$.
\end{proof}
\begin{figure}[t]
  \centering
  \begin{tikzpicture}[
    >=Stealth,
    state/.style={circle,draw,minimum size=18pt,inner sep=1pt},
    every edge/.style={draw,->,line width=0.6pt}
  ]
  \def\n{6} 
  \pgfmathtruncatemacro{\last}{\n-1}

  \foreach \i in {0,...,\last}{
    \node[state] (q\i) at ({360/\n * \i}:2.2cm) {$\i$};
  }

  \pgfmathtruncatemacro{\prelast}{\n-2}
  \foreach \i in {1,...,\prelast}{
    \pgfmathtruncatemacro{\j}{\i+1}
    \path (q\i) edge (q\j);
  }
  \path (q\last) edge (q0);

  \foreach \i in {1,...,\last}{
    \path (q\i) edge[loop above] (q\i);
  }

  \path (q0) edge[bend left=18] (q1);
  \path (q0) edge[bend right=18] (q1);

  \end{tikzpicture}
  \caption{Underlying 2-out graph of the Cerny's automaton.}
  \label{fig:cerny-underlying}
\end{figure}
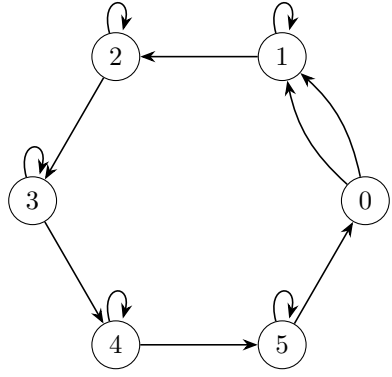
It would be interesting to quantify, for fixed parameters $n$ and $k$, the probability that a uniformly random $k$-out graph on
$n$ vertices admits a nontrivial lumpable partition, i.e.\ fails to be totally simple.
For fixed $k\ge 2$, such a uniformly random $k$-out digraph is typically not strongly connected.
Indeed, in our model each vertex chooses independently $k$ outgoing edges with replacement, and the head of each edge is uniform in $V$.
Hence, for a fixed vertex $v$ one has
\[
\Pr(\delta_{V}(v)=\emptyset)=\Bigl(1-\frac{1}{n}\Bigr)^{kn}\sim e^{-k},
\]
so the expected number of vertices with in-degree $0$ is $\sim n e^{-k}$, and therefore $G$ fails to be strongly connected with high
probability as $n\to\infty$. Moreover, if a $k$-out digraph $G$ is not strongly connected, then it is automatically nontrivially lumpable. Therefore, conditioning on strong connectivity is natural when studying typical total simplicity. We formulate the following conjecture.

\begin{conjecture}\label{conj:typical-total-simplicity}
Fix $k\ge 2$. Let $G$ be a uniformly random $k$-out graph on $n$ vertices, and write $\mathrm{SCC}$ for the event that $G$
is strongly connected. Then
\[
\Pr(\text{$G$ is totally simple}\mid \mathrm{SCC}) \xrightarrow[n\to\infty]{} 1.
\]
Moreover, for fixed $n$ the function $k \longmapsto \Pr(\text{$G$ is totally simple}\mid \mathrm{SCC})$ is nondecreasing. More quantitatively, there exist constants $c_k>0$ and $\beta_k\in\mathbb R$ such that
\[
\Pr(\text{$G$ is totally simple}\mid \mathrm{SCC})
\;=\;
1-\exp(-c_k n+\beta_k)+o(1)
\qquad (n\to\infty),
\]
with $c_k$ increasing as $k$ increases.
\end{conjecture}

We generated uniformly random $k$-out graphs, conditioned on strong connectivity, and tested whether they admit any nontrivial lumpable partition. Table~\ref{tab:totally-simple-experiments-agg} reports a representative subset of the resulting empirical proportions of totally simple instances, together with Wilson 95\% confidence intervals; the remaining experiments (not shown) exhibit the same qualitative behaviour and are consistent with the figures in the table.
Across the tested parameters, the observed probability of total simplicity increases with $n$ for fixed $k$, and increases
with $k$ for fixed $n$. Guided by numerical fits for $k=3,4$ on the tested range, we propose the following simple rational-valued ansatz:
\[
c_k \approx \frac{2^{k-2}}{25},
\qquad
\beta_k \approx \frac{3^{k-2}}{6}.
\]
In particular, $c_k$ approximately doubles when $k$ increases by one, while $\beta_k$ grows faster (roughly by a factor of $3$).
For $k=2$ the empirical estimates increase much more slowly on the tested range, and we therefore refrain from drawing
quantitative conclusions from an exponential fit in that case.
\begin{table}[t]
\centering
\begin{tabular}{|r|r|r|r|l|}
\hline
$n$ & $k$ & samples & totally simple & $\hat p$ (Wilson 95\% CI) \\
\hline
8  & 2 & 2000 & 691 & $0.346\;[0.325,\,0.367]$ \\
9  & 2 & 2000 & 722 & $0.361\;[0.340,\,0.382]$ \\
10 & 2 & 2000 & 692 & $0.346\;[0.325,\,0.367]$ \\
\hline
8  & 3 & 2000 & 1105 & $0.553\;[0.531,\,0.574]$ \\
9  & 3 & 2000 & 1262 & $0.631\;[0.610,\,0.652]$ \\
10 & 3 & 2000 & 1289 & $0.644\;[0.623,\,0.665]$ \\
\hline
8  & 4 & 2000 & 1453 & $0.727\;[0.707,\,0.746]$ \\
9  & 4 & 2000 & 1578 & $0.789\;[0.771,\,0.806]$ \\
10 & 4 & 2000 & 1659 & $0.830\;[0.812,\,0.845]$ \\
\hline
\end{tabular}
\caption{Empirical estimates of $\Pr(G\ \text{is totally simple}\mid \mathrm{SCC})$ for random $k$-out graphs,
conditioned on strong connectivity. Confidence intervals are Wilson 95\%, aggregated over two independent batches of $1000$ trials each.}
\label{tab:totally-simple-experiments-agg}
\end{table}

\subsection{Totally synchronizing and totally simple}
We now compare total simplicity and total synchronization.
By Theorem~\ref{theo: lump=to simple}, total simplicity is equivalent to the absence of nontrivial lumpable partitions.

First, total simplicity restricts endomorphisms: every endomorphism is an automorphism, and its vertex-orbit partition is either discrete or trivial.
Hence, if $|V(G)|$ is composite, then $\End(G)=\End_E(G)$.
Second, we prove a slight extension of a Perron--Frobenius criterion of Friedman and Gusev--Pribavkina: a strongly connected non-lumpable graph whose integer-normalized left Perron--Frobenius eigenvector admits at most one nontrivial equipartition is totally synchronizing.

\begin{remark}\label{rem:totally-simple-strongly-connected}
Every totally simple $k$-out graph is strongly connected. Indeed, if $G$ is not strongly connected,
let $C\subsetneq V(G)$ be a sink strongly connected component. For every coloring $\chi$ of $G$,
the set $C$ is closed under all letters. Hence the equivalence relation whose only non-singleton
class is $C$ is a nontrivial congruence of $\chi(G)$. Thus $\chi(G)$ is not simple, so $G$ is not
totally simple.
\end{remark}

\begin{proposition}\label{prop:totally-simple-endomorphisms}
Let $G$ be a totally simple $k$-out graph. Then \(\End(G)=\Aut(G).\) Moreover, if $\varphi\in\End(G)$, then exactly one of the following holds:
\begin{itemize}
\item $\varphi$ fixes every vertex;
\item $\varphi$ acts transitively on $V(G)$.
\end{itemize}
In particular, if $|V(G)|$ is composite, then $\End(G)=\End_E(G)$.
If, moreover, $G$ is parallel-free, then $\End(G)=\{\id\}$.
\end{proposition}
\begin{proof}
By Remark~\ref{rem:totally-simple-strongly-connected}, the graph $G$ is strongly connected.
Hence, by Proposition~\ref{prop:strongly-connected-endo-auto}, every endomorphism of $G$ is an automorphism. Let $\varphi\in\End(G)=\Aut(G)$, and let $\mathcal O(\varphi)$ be the partition of $V(G)$ into
vertex-orbits of $\varphi$. We show that $\mathcal O(\varphi)$ is lumpable. Let $O_i,O_j\in
\mathcal O(\varphi)$ and let $u,u'\in O_i$. Then $u'=\varphi^r(u)$ for some $r\ge0$. Since
$\varphi$ is a $k$-out automorphism and preserves $O_j$ setwise, for every $x\in V(G)$ one has
$|\delta_x(O_j)|=|\delta_{\varphi(x)}(O_j)|$. Iterating gives
\[
|\delta_u(O_j)|=|\delta_{\varphi^r(u)}(O_j)|=|\delta_{u'}(O_j)|.
\]
Thus $\mathcal O(\varphi)$ is lumpable. Since $G$ is totally simple, $\mathcal O(\varphi)$ is either the discrete partition or the trivial
partition $\{V(G)\}$. In the first case, $\varphi$ fixes every vertex. In the second case,
$\varphi$ acts transitively on $V(G)$.
\\
Assume now that $n:=|V(G)|$ is composite. If some $\varphi\in\End(G)$ acts transitively on $V(G)$,
then $\varphi_v$ is an $n$-cycle. Choose a divisor $d$ of $n$ with $1<d<n$. Then $\varphi^d$ is a
nontrivial automorphism, and the vertex-orbits of $\varphi^d$ are exactly $d$ orbits of cardinality
$n/d$. Hence this orbit partition is strictly between the discrete partition and $\{V(G)\}$.
By the first part of the proof, this partition is lumpable, contradicting total simplicity.
Therefore no endomorphism acts transitively on $V(G)$, and every endomorphism fixes every vertex.
Thus $\End(G)=\End_E(G)$.

If, moreover, $G$ is parallel-free, then every endomorphism fixing all vertices fixes every edge,
since there is at most one edge with any prescribed source and target. Hence $\End_E(G)=\{\id\}$,
and therefore $\End(G)=\{\id\}$.
\end{proof}
\begin{proposition}\label{prop:unique-equipartition}
Let $G$ be a finite strongly connected $k$-out directed graph which admits at least one non-synchronizing coloring.
Let $\pi\in\mathbb N^{V(G)}_{>0}$ be the integer-normalized left Perron--Frobenius eigenvector of $G$.
Fix a non-synchronizing coloring $\chi$ of $G$, and let
$\rho=\{B_1,\dots,B_m\}$ be the partition of $V(G)$ into maximal synchronizing subsets of the automaton $\chi(G)$.
Then the following are equivalent:
\begin{enumerate}
\item\label{it:1}
$\rho\in\mathrm{Cong}(\chi(G))$.
\item\label{it:2}
$\rho$ is the unique nontrivial equipartition of $\pi$ formed by maximal synchronizing subsets of $\chi(G)$.
\end{enumerate}
Moreover, if these conditions hold, then $\rho$ is lumpable for $G$ and the quotient $G/\rho$ is Eulerian.
If, in addition, $\rho$ is the unique nontrivial equipartition of $\pi$, then $|V(G/\rho)|=m$ is prime.
\end{proposition}
\begin{proof}
The equivalence between \ref{it:1} and \ref{it:2} is exactly \cite[Lemma~12]{GuPri}, applied to the non-synchronizing automaton $\chi(G)$.
Assume that these equivalent conditions hold. Since $\rho\in\mathrm{Cong}(\chi(G))$, Lemma~\ref{lem: lump=congr} implies that $\rho$ is lumpable for the underlying graph $G$. Hence the quotient $G/\rho$ is a well-defined $k$-out graph.
Let $\pi'$ be the integer-normalized left Perron--Frobenius eigenvector of $G/\rho$. By Lemma~\ref{lem:pf-lump-part}, $\pi'$ is proportional to the block-sums $\sum_{x\in B_i}\pi(x)$. Since $\rho$ is an equipartition of $\pi$, these sums are constant, hence $\pi'=w\mathbf 1$ for some $w>0$. For a $k$-out digraph, a constant positive left eigenvector for eigenvalue $k$ is equivalent to being Eulerian, since $(\mathbf 1^\top A')_j$ is the in-degree of the $j$-th quotient vertex. Thus $G/\rho$ is Eulerian. Assume finally that $\rho$ is the unique nontrivial equipartition of $\pi$. We prove that $m$ is prime. If $m=|V(G/\rho)|$ were composite, say $m=ab$ with $1<a,b<m$, then the vertex set of $G/\rho$ could be partitioned into $a$ subsets of size $b$. Pulling this partition back along the quotient map gives an equipartition of $\pi$ strictly coarser than $\rho$, because all $\rho$-blocks have the same $\pi$-weight. This contradicts uniqueness. Hence $m$ is prime.
\end{proof}

\begin{figure}[t]
\centering

\begin{subfigure}[t]{0.49\textwidth}
\centering
\begin{tikzpicture}[
  scale=0.88, transform shape,
  >=Latex,
  node/.style={circle, draw, thick, minimum size=7mm, inner sep=1pt},
  aedge/.style={->, thick, shorten <=2pt, shorten >=2pt},
  bedge/.style={->, thick, shorten <=2pt, shorten >=2pt}
]
\def\R{2.5}
\node[node] (v0) at (90:\R)  {$0$};
\node[node] (v1) at (30:\R)  {$1$};
\node[node] (v2) at (-30:\R) {$2$};
\node[node] (v3) at (-90:\R) {$3$};
\node[node] (v4) at (-150:\R){$4$};
\node[node] (v5) at (150:\R) {$5$};

\draw[aedge] (v0) to[bend left=14] (v2);
\draw[bedge] (v0) to[bend left=18] (v4);

\draw[aedge] (v1) to[bend right=18] (v0);
\draw[bedge] (v1) to[bend left=18] (v3);

\draw[aedge] (v2) to[bend right=12] (v1);
\draw[bedge] (v2) to[bend left=18] (v4);

\draw[aedge] (v3) to[bend left=20] (v1);
\draw[bedge] (v3) to[bend right=20] (v2);

\draw[aedge] (v4) to[bend left=16] (v0);
\draw[bedge] (v4) to[bend left=14] (v5);

\draw[aedge] (v5) to[bend left=18] (v0);
\draw[bedge] (v5) to[bend right=18] (v3);
\end{tikzpicture}
\caption{Non-lumpable and totally synchronizing: $\pi=(9,6,7,5,8,4)$ has the unique nontrivial equipartition $\{\{0,5\},\{1,2\},\{3,4\}\}$, so Theorem~\ref{theo: t-s for unique partition} applies.}
\label{fig:nonlumpable-ts}
\end{subfigure}
\hfill
\begin{subfigure}[t]{0.49\textwidth}
\centering
\begin{tikzpicture}[
  scale=0.88, transform shape,
  >=Latex,
  node/.style={circle, draw, thick, minimum size=7mm, inner sep=1pt},
  edge/.style={->, thick, shorten <=2pt, shorten >=2pt}
]
\def\R{2.5}
\node[node] (v0) at (90:\R)  {$0$};
\node[node] (v1) at (30:\R)  {$1$};
\node[node] (v2) at (-30:\R) {$2$};
\node[node] (v3) at (-90:\R) {$3$};
\node[node] (v4) at (-150:\R){$4$};
\node[node] (v5) at (150:\R) {$5$};

\draw[edge] (v0) edge[loop above, looseness=10] (v0);
\draw[edge] (v2) edge[loop right, looseness=10] (v2);
\draw[edge] (v4) edge[loop left,  looseness=10] (v4);

\draw[edge] (v0) to[bend left=12] (v3);

\draw[edge] (v1) to[bend left=14] (v2);
\draw[edge] (v1) to[bend left=14] (v5);

\draw[edge] (v2) to[bend left=12] (v1);

\draw[edge] (v3) to[bend left=10] (v1);
\draw[edge] (v3) to[bend right=12] (v5);

\draw[edge] (v4) to[bend left=12] (v0);

\draw[edge] (v5) to[bend left=12] (v3);
\draw[edge] (v5) to[bend right=12] (v4);
\end{tikzpicture}
\caption{Eulerian, strongly connected, non-lumpable, but not totally synchronizing (found by exhaustive search).}
\label{fig:nonlumpable-eul-nts}
\end{subfigure}

\vspace{0.6em}

\begin{subfigure}[t]{0.72\textwidth}
\centering
\begin{tikzpicture}[
  scale=0.82, transform shape,
  >=Latex,
  node/.style={circle, draw, thick, minimum size=7mm, inner sep=1pt},
  aedge/.style={->, thick, blue, shorten <=2pt, shorten >=2pt},
  bedge/.style={->, thick, red,  shorten <=2pt, shorten >=2pt},
  alab/.style={font=\small, inner sep=1pt, fill=white, fill opacity=0.85, text opacity=1},
  blab/.style={font=\small, inner sep=1pt, fill=white, fill opacity=0.85, text opacity=1}
]
\def\R{2.6}
\node[node] (v0) at (90:\R)  {$0$};
\node[node] (v1) at (18:\R)  {$1$};
\node[node] (v2) at (-54:\R) {$2$};
\node[node] (v3) at (-126:\R){$3$};
\node[node] (v4) at (162:\R) {$4$};

\draw[aedge] (v0) edge[loop above, looseness=10]
  node[alab, above, yshift=2pt] {$a$} (v0);
\draw[bedge] (v0) to[bend left=10]
  node[blab, pos=0.55, sloped, above, yshift=2pt] {$b$} (v4);

\draw[aedge] (v1) to[bend right=14]
  node[alab, pos=0.55, sloped, above, yshift=2pt] {$a$} (v4);
\draw[bedge] (v1) to[bend left=18]
  node[blab, pos=0.55, sloped, above, yshift=2pt] {$b$} (v3);

\draw[aedge] (v2) to[bend left=16]
  node[alab, pos=0.55, sloped, above, yshift=2pt] {$a$} (v3);
\draw[bedge] (v2) to[bend right=18]
  node[blab, pos=0.55, sloped, above, yshift=2pt] {$b$} (v0);

\draw[aedge] (v3) to[bend left=18]
  node[alab, pos=0.55, sloped, above, yshift=2pt] {$a$} (v1);
\draw[bedge] (v3) to[bend right=14]
  node[blab, pos=0.55, sloped, above, yshift=2pt] {$b$} (v0);

\draw[aedge] (v4) to[bend left=16]
  node[alab, pos=0.5, sloped, above, yshift=2pt] {$a$} (v2);
\draw[bedge] (v4) to[bend right=18]
  node[blab, pos=0.55, sloped, above, yshift=2pt] {$b$} (v3);
\end{tikzpicture}
\caption{A non-synchronizing automaton whose underlying graph is strongly connected and non-lumpable, with
left Perron--Frobenius eigenvector $\pi=(3,1,1,2,2)$ (which admits exactly two nontrivial partitions).}
\label{fig:nonlumpable-nonEuler-nts}
\end{subfigure}

\caption{Three non-lumpable strongly connected examples: a totally synchronizing one (left), an Eulerian non-totally-synchronizing one (right), and a non-Eulerian non-totally-synchronizing one (bottom).}
\label{fig:nonlumpable-examples}
\end{figure}
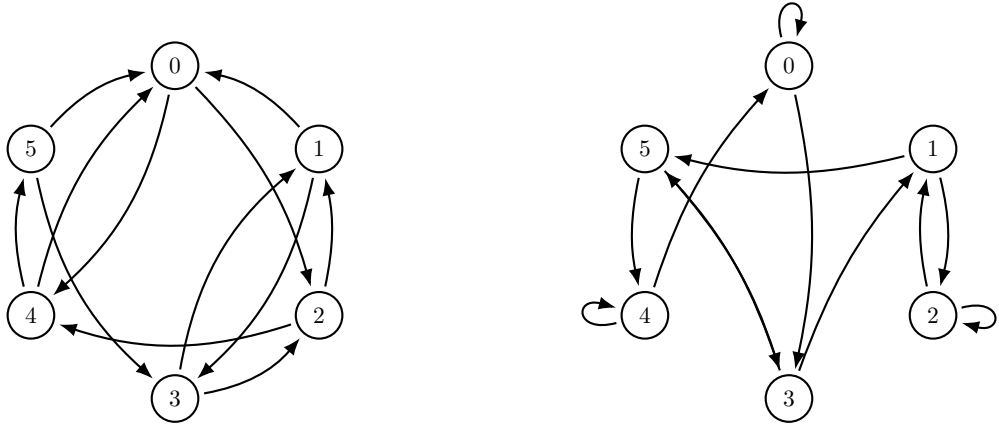
\begin{theorem}\label{theo: t-s for unique partition}
Let $G$ be a strongly connected non-lumpable $k$-out graph, and let
$\pi\in\mathbb N^{V(G)}_{>0}$ be its integer-normalized left Perron--Frobenius eigenvector.
If $\pi$ admits at most one nontrivial equipartition, then $G$ is totally synchronizing.
\end{theorem}
\begin{proof}
If $\pi$ admits no nontrivial equipartition, then $G$ is totally synchronizing by~\cite[Theorem~8]{GuPri}.
Assume that $\pi$ admits a unique nontrivial equipartition.
Suppose, by contradiction, that $G$ is not totally synchronizing. Then there exists a
non-synchronizing coloring $\chi$ of $G$. Let $\rho$ be the partition of $V(G)$ into maximal
synchronizing subsets of the automaton $\chi(G)$.
Since $\chi(G)$ is not synchronizing, $\rho$ is nontrivial. By Theorem~\ref{thm:friedman-weight},
$\rho$ is an equipartition of $\pi$. Hence $\rho$ is the unique nontrivial equipartition of $\pi$.
By Proposition~\ref{prop:unique-equipartition}, applied to $\chi$ and $\rho$, the partition
$\rho$ is lumpable for $G$. This contradicts the assumption that $G$ is non-lumpable.
Therefore every coloring of $G$ is synchronizing.
\end{proof}
We verified by exhaustive computation that all three underlying digraphs in Figure~\ref{fig:nonlumpable-examples} are non-lumpable.
The left example is nevertheless totally synchronizing: its integer-normalized Perron--Frobenius eigenvector admits a unique
nontrivial equipartition, hence Theorem~\ref{theo: t-s for unique partition} applies.
In contrast, the right (Eulerian) and bottom (non-Eulerian) examples admit non-synchronizing colorings, and therefore are not totally synchronizing.
Taken together, these examples show that non-lumpability, although a rather strong structural constraint, does not imply total synchronization, and thus motivate the following question.

\begin{open}\label{open:char-tot-simple-tot-sync}
Characterize totally simple strongly connected $k$-out graphs that are totally synchronizing.
\end{open}
\section{Complexity results}\label{sec: complexity}
In this section we study computational problems related to total synchronization and lumpability.
We prove that deciding whether a primitive $k$-out graph admits a non-synchronizing coloring is NP-complete, resolving an open problem of Gusev and Szyku{\l}a~\cite[Problem~1]{GusevSzykula15}.
We also prove NP-completeness for the existence of a nontrivial Eulerian lumping, and give a reduction from counting labeled $1$-factorizations to counting non-synchronizing colorings.

For decision problems with integer inputs, \emph{strong NP-hardness} means that NP-hardness persists when numbers are encoded in unary.
A problem is \emph{strongly NP-complete} if it is in NP and strongly NP-hard; see \cite[Ch.~4]{GareyJohnson}.
\begin{theorem}\label{thm:exists-m-strong-npc}
\textsc{Exists-Equipartition} is strongly NP-complete:

\medskip
\noindent\textbf{Instance:} positive integers $a_1,\dots,a_n$;

\noindent\textbf{Question:} is the vector $(a_1,\ldots, a_n)$ equipartitionable into $m\ge 2$ components?
\end{theorem}

\begin{proof}
Membership in NP is immediate: a certificate is $(m,S_1,\dots,S_m)$, and one checks in polynomial time
that the $S_j$ form a partition and that all block sums are equal.

For strong NP-hardness we reduce from restricted \textsc{3-Partition}, which is strongly NP-complete
\cite[SP15]{GareyJohnson}. Thus we are given integers $(b_1,\dots,b_{3m};B)$ such that
\[
\sum_{i=1}^{3m} b_i = mB
\qquad\text{and}\qquad
\frac{B}{4}<b_i<\frac{B}{2}\ \ \forall i,
\]
and we ask whether $\{1,\dots,3m\}$ can be partitioned into $m$ triples each summing to $B$.

Choose a prime number $p$ with $m<p<2m$, which exists by the Bertrand--Chebyshev theorem
\cite{ApostolANT}. (Here $m$ is determined by the input size, since the instance contains
exactly $3m$ integers.) Define the marker value
\[
Q := pB + 1,
\]
and construct the multiset
\[
\mathcal A \;:=\; \{b_1,\dots,b_{3m}\}\ \cup\ \{\underbrace{B,\dots,B}_{p-m\ \text{copies}}\}\ \cup\
\{\underbrace{Q,\dots,Q}_{p\ \text{copies}}\}.
\]
Its total sum is
\[
T \;=\; mB + (p-m)B + pQ \;=\; pB + pQ \;=\; p(B+Q).
\]

We claim that the \textsc{3-Partition} instance is a YES-instance if and only if $\mathcal A$ admits an
equal-sum partition into some number of blocks $m''\ge 2$.

If the \textsc{3-Partition} instance is YES, partition the $b_i$ into $m$ triples $T_1,\dots,T_m$ each summing to $B$.
Then form $p$ blocks of sum $B+Q$ by taking
\[
T_r\cup\{Q\}\quad(r=1,\dots,m),
\qquad\text{and}\qquad
\{B\}\cup\{Q\}\quad\text{for the remaining }(p-m)\text{ blocks}.
\]
Hence $\mathcal A$ is a YES-instance of \textsc{Exists-Equipartition}.

Conversely, suppose $\mathcal A$ admits an equal-sum partition into $m''\ge2$ blocks, each of common sum $N$.
Write the elements of $\mathcal A$ as three groups:
\[
\underbrace{b_1,\dots,b_{3m}}_{\text{original}},\qquad
\underbrace{B,\dots,B}_{p-m\ \text{fillers}},\qquad
\underbrace{Q,\dots,Q}_{p\ \text{markers}}.
\]
The total contribution of the \emph{non-marker} elements (i.e.\ original plus fillers) is
\[
\sum_{i=1}^{3m} b_i \;+\; (p-m)\,B \;=\; mB + (p-m)B \;=\; pB.
\]
Now fix any block in the partition and let it contain $r$ markers. The remaining (non-marker) elements
in that block have sum exactly $N-rQ$, and since all non-marker elements in the \emph{entire instance}
sum to $pB$, we always have
\[
0 \le N-rQ \le pB.
\]
In particular, if two blocks contained different numbers of markers, say $r$ and $r+1$, then the sums of
their non-marker parts would differ by exactly $Q$: $(N-rQ) - (N-(r+1)Q) = Q.$
But each of these non-marker sums lies in $[0,pB]$, and $Q=pB+1$, so such a difference is impossible.
Hence every block contains the same number $r$ of markers, and therefore $m''r=p$, that is, $m''=p$, $r=1$. Hence, we have  $N=T/p=B+Q$. Removing the unique marker from each block leaves $p$ blocks whose sums are all $B$. Each filler element equals $B$, so it must appear alone in its post-marker block (otherwise the sum would exceed $B$).
Thus the remaining $m$ post-marker blocks consist only of the original numbers $b_i$ and each has sum $B$.
By the bounds $\frac{B}{4}<b_i<\frac{B}{2}$, each such block contains exactly three elements, yielding a valid
\textsc{3-Partition} solution. Finally, the reduction adds $O(p)=O(m)$ extra integers, and the largest new integer is $Q=pB+1$ with $p<2m$,
so all numbers produced are polynomially bounded in the size of the restricted \textsc{3-Partition} instance.
Since restricted \textsc{3-Partition} is strongly NP-complete , this proves strong NP-hardness.
Therefore \textsc{Exists-Equipartition} is strongly NP-complete.
\end{proof}
Let $\mathbf w=(w_1,\dots,w_n)\in\mathbb N^n_{>0}$, and set $W:=\sum_{j=1}^n w_j$. Define a graph $G(\mathbf w)=(V,E,s,t)$ by setting $V=[n]=\{1,\dots,n\}$ and
\[
E=\{e_{i,j,r}: i,j\in[n],\ 1\le r\le w_j\},
\]
with $s(e_{i,j,r})=i$ and $t(e_{i,j,r})=j$. Thus, for each ordered pair $(i,j)$ there are exactly $w_j$ parallel edges from $i$ to $j$. In particular, every vertex has out-degree $W$, so $G(\mathbf w)$ is a $W$-out graph. Moreover, since $w_j>0$ for every $j$, there is at least one edge from every vertex to every vertex, and therefore $G(\mathbf w)$ is primitive. We call $G(\mathbf w)$ the \emph{maximally lumpable} digraph associated to $\mathbf w$; this construction already appears in \cite[Theorem~8]{GuPri}.
\begin{lemma}\label{lem:every-partition-lumpable}
Every partition $\mathcal C=\{C_1,\dots,C_m\}$ of $V=[n]$ is lumpable for $G(\mathbf w)$. Moreover, $\mathbf w$ is a left Perron--Frobenius eigenvector of $G(\mathbf w)$.
\end{lemma}
\begin{proof}
Fix blocks $C_a,C_b\in\mathcal C$ and vertices $u,u'\in C_a$. By construction, \(
|\delta_u(C_b)|=\sum_{j\in C_b} w_j=|\delta_{u'}(C_b)|\).
Hence $\mathcal C$ is lumpable. Now let $A$ be the adjacency matrix of $G(\mathbf w)$, so $A_{ij}=w_j$ for all $i,j$, and set $W:=\sum_{i=1}^n w_i$. Then for each $j$,
\[
(\mathbf w^\top A)_j=\sum_{i=1}^n w_iA_{ij}=\sum_{i=1}^n w_i\,w_j=Ww_j.
\]
Thus $\mathbf w^\top A=W\mathbf w^\top$. Since $A\ge 0$ and $\mathbf w>0$, it follows by Perron--Frobenius theorem that $W$ is the Perron--Frobenius eigenvalue and $\mathbf w$ is a corresponding left Perron--Frobenius eigenvector.
\end{proof}

\begin{lemma}\label{lem:eulerian-for-Gw-equivalent-to-partitionable}
Let $\mathbf w=(w_1,\dots,w_n)\in\mathbb N^n$ and let $G(\mathbf w)=(V,E,s,t)$ be the maximally lumpable digraph associated to $\mathbf w$, with $V=[n]$. Let $\mathcal C=\{C_1,\dots,C_m\}$, with $m\ge 2$, be a partition of $V$, and let $\rho_{\mathcal C}$ be the associated congruence. Then $G(\mathbf w)/\rho_{\mathcal C}$ is nontrivial and Eulerian if and only if $\mathcal C$ is an equipartition of $\mathbf w$. In particular, $G(\mathbf w)$ admits a nontrivial Eulerian quotient if and only if $\mathbf w$ is equipartitionable.
\end{lemma}
\begin{proof}
Set $W:=\sum_{v=1}^n w_v$. By Lemma~\ref{lem:every-partition-lumpable}, the partition $\mathcal C$ is lumpable, so the quotient $G(\mathbf w)/\rho_{\mathcal C}$ is well defined. Fix blocks $C_i,C_j\in\mathcal C$ and a vertex $u\in C_i$. By construction, for each $v\in[n]$ there are exactly $w_v$ edges from $u$ to $v$, hence $|\delta_u(C_j)|=\sum_{v\in C_j} w_v$. Therefore the number of edges from the quotient vertex $C_i$ to the quotient vertex $C_j$ is $\sum_{v\in C_j} w_v$, independently of $i$.

It follows that every quotient vertex has  $W$, while the in-degree of $C_i$ is $m\sum_{v\in C_i} w_v$. Hence $G(\mathbf w)/\rho_{\mathcal C}$ is Eulerian if and only if $W=m\sum_{v\in C_i} w_v$ for every $i$, equivalently $\sum_{v\in C_i} w_v=W/m$ for every $i$. This is exactly the condition that $\mathcal C$ be an equipartition of $\mathbf w$.
\end{proof}
\begin{proposition}\label{prop:not-totally-synchronizing}
Let $\mathbf w\in\mathbb N_{>0}^{n}$. Then $G(\mathbf w)$ is not totally synchronizing if and only if $\mathbf w$ is equipartitionable.
\end{proposition}
\begin{proof}
Assume first that $\mathbf w$ admits an equipartition $\mathcal C$. By Lemma~\ref{lem:eulerian-for-Gw-equivalent-to-partitionable}, the quotient graph $G(\mathbf w)/\rho_{\mathcal C}$ is Eulerian and nontrivial. Hence, by \cite[Lemma~1]{Kari2003}, it admits a non-synchronizing coloring $\chi$. By Lemma~\ref{lem:lifting-coloring}, $\chi$ lifts to a coloring $\widetilde\chi$ of $G(\mathbf w)$ such that $\rho_{\mathcal C}\in\Cong(\widetilde\chi(G(\mathbf w)))$ and the quotient automaton $\widetilde\chi(G(\mathbf w))/\rho_{\mathcal C}$ is exactly $\chi(G(\mathbf w)/\rho_{\mathcal C})$. Therefore $\widetilde\chi(G(\mathbf w))$ is not synchronizing, since otherwise its quotient would be synchronizing as well. Thus $G(\mathbf w)$ is not totally synchronizing. Conversely, suppose that $G(\mathbf w)$ is not totally synchronizing, and let $\chi$ be a coloring such that $\chi(G(\mathbf w))$ is not synchronizing. Since $G(\mathbf w)$ is primitive, and thus, strongly connected, Theorem~\ref{thm:friedman-weight} applies, and by Lemma~\ref{lem:every-partition-lumpable} the vector $\mathbf w$ is a left Perron--Frobenius eigenvector of $G(\mathbf w)$. Hence there exists a partition of $V(G(\mathbf w))$ into synchronizing subsets of equal maximal $\mathbf w$-weight. Because $\chi(G(\mathbf w))$ is not synchronizing, this partition is nontrivial, say $S_1,\dots,S_m$ with $m\ge 2$. Therefore $\mathbf w$ is equipartitionable.
\end{proof}
The following theorem shows that deciding whether a $k$-out graph $G$ has a non-synchronizing coloring is NP-complete, thereby solving \cite[Problem~1]{GusevSzykula15}.
\begin{theorem}\label{thm:NTS-NPcomplete}
\textsc{Not-Totally-Synchronizing} is NP-complete:

\smallskip
\noindent\textbf{Instance:} a primitive $k$-out graph $G$;

\noindent\textbf{Question:} does there exist a coloring $\chi$ such that the colored automaton $\chi(G)$ is not synchronizing?
\end{theorem}
\begin{proof}
Membership in NP is immediate: a certificate is a coloring $\chi$, and one can check in polynomial time whether $\chi(G)$ is non-synchronizing, for instance via the standard reachability test on the pair automaton.
For NP-hardness, we give a polynomial-time many-one reduction from \textsc{Exists-Equipartition}. Since \textsc{Exists-Equipartition} is strongly NP-complete, it remains NP-complete when the input vector $\mathbf w=(a_1,\dots,a_n)$ is given in unary. From such an instance we can therefore construct in polynomial time the maximally lumpable graph $G(\mathbf w)$. By construction, $G(\mathbf w)$ is primitive. By Proposition~\ref{prop:not-totally-synchronizing}, the graph $G(\mathbf w)$ is not totally synchronizing if and only if $\mathbf w$ is equipartitionable. Therefore $G(\mathbf w)$ is a YES-instance of \textsc{Not-Totally-Synchronizing} if and only if $\mathbf w$ is a YES-instance of \textsc{Exists-Equipartition}. This proves NP-hardness, and hence NP-completeness.
\end{proof}
In the previous reduction the out-degree of $G(\mathbf w)$ is $k=\sum_i w_i$, hence it is not fixed. Thus the theorem resolves the general decision problem, while the fixed-out-degree case remains open. 
\begin{open}\label{open:nts-fixed-k}
Is \textsc{Not-Totally-Synchronizing} still NP-complete when the out-degree $k$ is fixed, for instance for $k=2$?
\end{open}
Beyond the decision problem, it is natural to ask for quantitative information on the number of synchronizing and non-synchronizing colorings of a primitive $k$-out graph. This point of view is closely related to the synchronizing ratio studied by Gusev and Szyku{\l}a, and leads naturally to the corresponding counting problem. For a $k$-out graph $G$, recall that $\mathcal C(G)$ denotes the set of all colorings, and define
\[
\#\mathrm{NSC}(G):=\bigl|\{\chi\in\mathcal C(G): \chi(G)\text{ is non-synchronizing}\}\bigr|,
\qquad
\#\mathrm{Sync}(G):=|\mathcal C(G)|-\#\mathrm{NSC}(G).
\]
Since $|\mathcal C(G)|=(k!)^{|V(G)|}$, exact counting of $\#\mathrm{NSC}(G)$ is equivalent to exact counting of $\#\mathrm{Sync}(G)$.

\begin{conjecture}[\cite{GusevCIRM2015}]\label{conj:count-nsc}
The counting problem $G\mapsto \#\mathrm{NSC}(G)$ is \#P-complete on the class of primitive $k$-out graphs. Moreover, this remains \#P-complete for fixed out-degree $k$, already for $k=2$.
\end{conjecture}
We were not able to prove Conjecture~\ref{conj:count-nsc}, but we can relate $\#\mathrm{NSC}$ to a classical counting problem on regular bipartite simple graphs. In a $k$-regular bipartite simple graph, a decomposition of the edge set into $k$ perfect matchings is the same object as a proper edge-$k$-coloring; such decompositions are also
known as $1$-factorizations. Exact counting of $1$-factorizations is widely believed to be computationally hard, and the precise complexity
status of counting edge-$k$-colorings of bipartite graphs (equivalently, $1$-factorizations in the $k$-regular case) is currently open~\cite{HongMiklos2023}. Our contribution here is a
tight reduction showing that, on prime-size strongly connected Eulerian graphs, counting non-synchronizing colorings
is essentially the same as counting permutation colorings, and the latter coincides with counting labeled
$1$-factorizations of the associated bipartite graph.

\medskip
Let $B=(U,W,E)$ be a $k$-regular bipartite (undirected) multigraph with $|U|=|W|$. A \emph{labeled $1$-factorization}
is an ordered $k$-tuple $(M_1,\dots,M_k)$ of pairwise edge-disjoint perfect matchings whose union is $E$.
Let $\#1\mathrm{Fact}_{\mathrm{lab}}(B)$ be the number of labeled $1$-factorizations of $B$.
\begin{lemma}\label{lem:bip-eul-correspondence}
Let $B=(U,W,E_B)$ be a $k$-regular bipartite multigraph, with $U=\{u_1,\dots,u_m\}$ and $W=\{w_1,\dots,w_m\}$. Construct the $k$-out directed multigraph $H_B=(V,E,s,t)$ with $V=\{v_1,\dots,v_m\}$ as follows: for every bipartite edge $e\in E_B$ with endpoints $u_i\in U$ and $w_j\in W$ (with multiplicities), add a directed edge $\phi(e)\in E$ and set $s(\phi(e))=v_i$ and $t(\phi(e))=v_j$. Then $H_B$ is $k$-out Eulerian and
\[
\#\mathrm{Perm}(H_B)=\#1\mathrm{Fact}_{\mathrm{lab}}(B),
\]
where $\#\mathrm{Perm}(H_B)$ denotes the number of permutation colorings of $H_B$, i.e.\ colorings $\chi:E\to[k]$ such that $\chi(H_B)$ acts like a permutation.
\end{lemma}

\begin{proof}
Since $B$ is $k$-regular, each $u_i$ is incident with exactly $k$ edges of $E_B$. By construction, every edge of $E_B$ incident to $u_i$ gives rise to exactly one edge of $E$ with source $v_i$, hence $|\delta_{v_i}(V)|=k$. Similarly, each $w_j$ is incident with exactly $k$ edges of $E_B$, and each such edge gives rise to exactly one edge of $E$ with target $v_j$, hence $|\delta_V(v_j)|=k$. Therefore $H_B$ is Eulerian.
Now fix a labeled $1$-factorization $(M_1,\dots,M_k)$ of $B$, and define a coloring $\chi:E\to[k]$ by setting $\chi(\phi(e))=\sigma$ if and only if $e\in M_\sigma$. For each $\sigma\in[k]$ and each $i\in[m]$, since $M_\sigma$ is a perfect matching, there is exactly one edge of $M_\sigma$ incident to $u_i$, hence exactly one edge in $\delta_{v_i}(V)$ colored $\sigma$. Likewise, there is exactly one edge of $M_\sigma$ incident to $w_i$, hence exactly one edge in $\delta_V(v_i)$ colored $\sigma$. Thus $\chi$ is a permutation coloring of $H_B$. Conversely, let $\chi:E\to[k]$ be a permutation coloring of $H_B$, and for each $\sigma\in[k]$ define $M_\sigma:=\{e\in E_B:\chi(\phi(e))=\sigma\}$. The condition on $\delta_{v_i}(V)$ implies that $M_\sigma$ contains exactly one edge incident to each $u_i$, and the condition on $\delta_V(v_i)$ implies that $M_\sigma$ contains exactly one edge incident to each $w_i$. Thus $M_\sigma$ is a perfect matching of $B$. Moreover, the sets $M_1,\dots,M_k$ are pairwise disjoint and their union is $E_B$, so $(M_1,\dots,M_k)$ is a labeled $1$-factorization. The two constructions are inverse to each other, establishing a bijection between permutation colorings of $H_B$ and labeled $1$-factorizations of $B$, and therefore $\#\mathrm{Perm}(H_B)=\#1\mathrm{Fact}_{\mathrm{lab}}(B)$.
\end{proof}
\begin{proposition}\label{prop:encode-1fact-into-nsc}
Fix $k\ge 3$. Assume that computing $\#1\mathrm{Fact}_{\mathrm{lab}}(B)$ is \#P-hard already on the class of $k$-regular bipartite multigraphs $B=(U,W,E)$ such that $|U|=|W|=p$ is prime and the associated directed multigraph $H_B$ from Lemma~\ref{lem:bip-eul-correspondence} is strongly connected. Then computing
\[
G\longmapsto \#\mathrm{NSC}(G)
\]
is \#P-hard on the class of strongly connected Eulerian $k$-out graphs with a prime number of vertices.
Consequently, under this assumption, the problem is \#P-complete on that class.
\end{proposition}

\begin{proof}
Membership in \#P is immediate, since for a given $k$-out graph $G$ and a coloring $\chi\in\mathcal C(G)$ one can check in polynomial time whether $\chi(G)$ is non-synchronizing. Now let $B=(U,W,E)$ be an instance satisfying the hypothesis, and construct $H_B$ as in Lemma~\ref{lem:bip-eul-correspondence}. By assumption, $H_B$ is strongly connected, and $|V(H_B)|=|U|=p$ is prime. Moreover, Lemma~\ref{lem:bip-eul-correspondence} gives $\#\mathrm{Perm}(H_B)=\#1\mathrm{Fact}_{\mathrm{lab}}(B)$. Since $H_B$ is strongly connected, Eulerian, and has a prime number of vertices, Proposition~\ref{lem:prime-eulerian-dichotomy} implies that for every coloring $\chi\in\mathcal C(H_B)$, the automaton $\chi(H_B)$ is non-synchronizing if and only if it is permutative. Hence $\#\mathrm{NSC}(H_B)=\#\mathrm{Perm}(H_B)=\#1\mathrm{Fact}_{\mathrm{lab}}(B)$. Thus computing $\#\mathrm{NSC}$ is \#P-hard already on the restricted class of strongly connected Eulerian $k$-out graphs with a prime number of vertices.
Since $\#\mathrm{NSC}$ belongs to \#P, the conditional \#P-completeness statement follows.
\end{proof}

\subsection{Complexity results for quotients of $k$-out graphs}
While lumpability and Markov chain aggregation have a long structural literature, the existence problem over all partitions is rarely treated from a worst-case complexity viewpoint. As a result of the previous techniques we have the following result. 
\begin{theorem}\label{thm:eulerian-lumping-npcomplete}\textsc{Exists-Eulerian-Lumping} is NP-complete:

\textbf{Instance:} a primitive $k$-out graph $G$.

\textbf{Question:} does there exist a nontrivial lumpable partition $\rho$ of $V(G)$ such that the quotient $k$-out graph
$G/\rho$ is Eulerian? 
\end{theorem}

\begin{proof}
A certificate is a partition $\rho=\{C_1,\dots,C_m\}$ with $2\le m\le |V(G)|-1$. In polynomial time one checks lumpability of $\rho$ (i.e.\ for all blocks $C_a,C_b$ and all $u,u'\in C_a$, the number of edges from $u$ to $C_b$ equals the number of edges from $u'$ to $C_b$). One also checks in polynomial time that $G/\rho$ is Eulerian by comparing in-degree and out-degree of each quotient vertex. For the NP-hardness we give a many-one reduction from \textsc{Exists-Equipartition} (Theorem~\ref{thm:exists-m-strong-npc}). Given $\mathbf w=(w_1,\dots,w_n)\in\mathbb N_{>0}^n$, construct in polynomial time the maximally lumpable digraph
$G(\mathbf w)$. By Lemma~\ref{lem:every-partition-lumpable}, \emph{every} partition $\mathcal C$ of $V(G(\mathbf w))=[n]$ is lumpable,
so the quotient $G(\mathbf w)/\rho_{\mathcal C}$ is always well-defined. By Lemma~\ref{lem:eulerian-for-Gw-equivalent-to-partitionable}, we have that checking whether $G(\mathbf w)/\rho_{\mathcal C}$ is an Eulerian non-trivial $k$-out graph is equivalent to checking whether $\mathbf w$ is equipartitionable. Therefore, there exists a nontrivial lumpable partition $\rho$ with Eulerian quotient if and only if $\mathbf w$ is equipartitionable. This proves NP-hardness.
\end{proof}
Equivalently, in Markov chain terminology, deciding whether there exists a nontrivial lumpable partition $\rho$ such that the lumped chain $P^\rho$ is doubly stochastic is NP-complete. Also in this case holds the same open problem to evaluate the computational complexity for a fixed $k$. 
\begin{remark}
While Theorem~\ref{thm:eulerian-lumping-npcomplete} yields NP-completeness for the existence of an Eulerian lumpable quotient,
the complexity of deciding whether a primitive $k$-out graph admits \emph{any} nontrivial lumpable quotient (without the
Eulerian constraint) seems to be open.
\end{remark}
A partition $\rho$ of $V(G)$ is lumpable if and only if the quotient $G/\rho$ is a well-defined $k$-out graph; equivalently, lumpability is the same as the existence of a surjective morphism of $k$-out graphs
$f:G\to H$ with kernel $\rho$ (so $H\simeq G/\rho$).
Hence asking whether $G$ admits a nontrivial lumpable partition different from the discrete and universal ones is
equivalent to asking whether $G$ has a \emph{nontrivial image} in the category of $k$-out graphs, i.e.\ whether there exists
a surjective morphism $f:G\to H$ with
\[
2\le |V(H)|\le |V(G)|-1.
\]
Theorem~\ref{thm:eulerian-lumping-npcomplete} settles a constrained version of this question: it is NP-complete to decide
whether $G$ has such a nontrivial image $H$ that is moreover Eulerian. This leaves open the problem of framing the worst-case computational complexity of deciding whether a given primitive $k$-out graph has any nontrivial image at all. In particular, by Theorem~\ref{theo: lump=to simple}, it remains open to classify the complexity of recognizing totally simple graphs.
\begin{open}
Is recognizing non-totally-simple $k$-out graphs NP-complete?
Equivalently, is deciding whether a primitive $k$-out graph admits a nontrivial lumpable partition NP-complete?
\end{open}
Moreover, the undirected analogue, deciding whether a connected $k$-regular simple graph admits a nontrivial equitable partition, has been studied recently under the more general notion of \emph{weight-equitable} partitions; see \cite{AbiadHojnyZeijlemakerWEP}.
If $X$ is $k$-regular simple graph and $\overleftrightarrow{X}$ denotes its bidirected orientation (replacing each edge $\{u,v\}$ with the pair $u\to v$ and $v\to u$), then equitable partitions of $X$ are in bijection with lumpable partitions of the $k$-out digraph $\overleftrightarrow{X}$.
In particular, the existence of a nontrivial equitable partition in $X$ is polynomial-time equivalent to the existence of a nontrivial lumpable partition in the restricted class of bidirected $k$-out digraphs.

\section*{Acknowledgements}
This paper is dedicated to the memory of the late Elena Pribavkina, whose work on synchronizing automata, the Černý conjecture, and totally synchronizing digraphs influenced the present study.
\\
The authors are members of the National Research Group GNSAGA
(Gruppo Nazionale per le Strutture Algebriche, Geometriche e le loro
Applicazioni) of INdAM.

\printbibliography

@article{AlRo,
author = {Almeida, J. and Rodaro, E.},
title = {Semisimple Synchronizing Automata and the Wedderburn-Artin Theory},
journal = {International Journal of Foundations of Computer Science},
volume = {27},
number = {02},
pages = {127-145},
year = {2016},
doi = {10.1142/S0129054116400037}
}

@article{BehJoh,
author = {Behague, N. and Johnson, R.},
year = {2022},
pages = {},
title = {Synchronizing Times for $k$-sets in Automata},
volume = {29},
journal = {Electronic Journal of Combinatorics},
doi = {10.37236/9819}
}

@article{BeBePe,
author = {B\'{e}al, M.-P. and Berlinkov, M. V. and Perrin, D.},
title = {A quadratic upper bound on the size of a synchronizing word  in one-cluster automata},
journal = {International Journal of Foundations of Computer Science},
volume = {22},
number = {02},
pages = {277-288},
year = {2011},
doi = {10.1142/S0129054111008039}
}

@article{Perrin_unamb_coded_shifts,
title = {Unambiguously coded shifts},
journal = {European Journal of Combinatorics},
volume = {119},
pages = {103812},
year = {2024},
note = {Dedicated to the memory of Pierre Rosenstiehl},
issn = {0195-6698},
doi = {https://doi.org/10.1016/j.ejc.2023.103812},
author = {Béal, M.-P. and Perrin, D. and Restivo, A.}
}

@article{Ce64,
author = {Černý, J.},
journal = {Matematicko-fyzikálny časopis},
language = {slo},
number = {3},
pages = {208-216},
publisher = {Mathematical Institute of the Slovak Academy of Sciences},
title = {Poznámka k homogénnym experimentom s konečnými automatmi},
volume = {14},
year = {1964},
}

@article{Don,
  title={The \v{C}ern{\'y} Conjecture and 1-Contracting Automata},
  author={Don, H.},
  journal={Electronic Journal of Combinatorics},
  year={2015},
  volume={23},
  pages={3}
}

@article{Dubuc,
	author = {Dubuc, L.},
	title = {Sur les automates circulaires et la conjecture de Černý},
	DOI= "10.1051/ita/1998321-300211",
	journal = {RAIRO-Theor. Inf. Appl.},
	year = 1998,
	volume = 32,
	number ="1-3",
	pages = "21-34",
}

@article{Epp,
author = {Eppstein, D.},
title = {Reset Sequences for Monotonic Automata},
journal = {SIAM Journal on Computing},
volume = {19},
number = {3},
pages = {500-510},
year = {1990},
doi = {10.1137/0219033}
}

@article{Kari,
title = {Synchronizing finite automata on Eulerian digraphs},
journal = {Theoretical Computer Science},
volume = {295},
number = {1},
pages = {223-232},
year = {2003},
note = {Mathematical Foundations of Computer Science},
issn = {0304-3975},
doi = {https://doi.org/10.1016/S0304-3975(02)00405-X},
author = {Kari, J.}
}

@article{Ro18Adv,
 author = {Rodaro, E.},
 title = {Strongly connected synchronizing automata and the language of minimal reset words},
 journal = {Advances in Applied Mathematics},
 %journal = {Adv. Appl. Math.},
 issn = {0196-8858},
 volume = {99},
 pages = {158--173},
 year = {2018},
 language = {English},
 doi = {10.1016/j.aam.2018.04.006}
}

@article{Steinb,
 author = {Steinberg, B.},
 title = {The {\v{C}}ern{\'y} conjecture for one-cluster automata with prime length cycle},
 journal = {Theoretical Computer Science},
 %journal = {Theor. Comput. Sci.},
 issn = {0304-3975},
 volume = {412},
 number = {39},
 pages = {5487--5491},
 year = {2011},
 language = {English},
 doi = {10.1016/j.tcs.2011.06.012}
}

@article{SteinbEJC,
 author = {Steinberg, B.},
 title = {A theory of transformation monoids: combinatorics and representation theory},
 journal = {Electronic Journal of Combinatorics},
 %journal = {Electron. J. Comb.},
 issn = {1077-8926},
 volume = {17},
 number = {1},
 pages = {research paper r164, 56},
 year = {2010},
 language = {English}
}

@article{Trah,
 author = {Trahtman, A. N.},
 title = {The {\v{C}}ern{\'y} conjecture for aperiodic automata},
 journal = {Discrete Mathematics and Theoretical Computer Science. DMTCS},
 %journal = {Discrete Math. Theor. Comput. Sci.},
 issn = {1365-8050},
 volume = {9},
 number = {2},
 pages = {3--10},
 year = {2007},
 language = {English}
}

@incollection{Vo_CIAA07,
 author = {Volkov, M. V.},
 title = {Synchronizing automata preserving a chain of partial orders},
 booktitle = {Implementation and application of automata. 12th international conference, CIAA 2007, Prague, Czech Republic, July 16--18, 2007. Revised selected papers},
 isbn = {978-3-540-76335-2},
 pages = {27--37},
 year = {2007},
 publisher = {Berlin: Springer},
 language = {English},
 doi = {10.1007/978-3-540-76336-9_5}
}

@misc{survey_volkov_2025,
      title={List of Results on the \v{C}ern\'y Conjecture and Reset Thresholds for Synchronizing Automata}, 
      author={Volkov, M. V.},
      year={2025},
      eprint={2508.15655},
      archivePrefix={arXiv},
      primaryClass={cs.FL},
      url={https://arxiv.org/abs/2508.15655}, 
}

@article{AbiadHojnyZeijlemakerWEP,
  author  = {Abiad, A. and Hojny, C. and Zeijlemaker, S.},
  title   = {Characterizing and computing weight-equitable partitions of graphs},
  journal = {Linear Algebra and its Applications},
  volume  = {645},
  pages   = {30--51},
  year    = {2022},
  doi     = {10.1016/j.laa.2022.03.003}
}

@article{adler1977,
  author  = {Adler, R. L. and Goodwyn, L. W. and Weiss, B.},
  title   = {Equivalence of topological Markov shifts},
  journal = {Israel Journal of Mathematics},
  volume  = {27},
  number  = {1},
  pages   = {49--63},
  year    = {1977}
}

@book{ApostolANT,
  author    = {Apostol, T. M.},
  title     = {Introduction to Analytic Number Theory},
  series    = {Undergraduate Texts in Mathematics},
  publisher = {Springer-Verlag},
  address   = {New York},
  year      = {1976}
}

@article{Buchholz1994,
  author  = {Buchholz, P.},
  title   = {Exact and ordinary lumpability in finite Markov chains},
  journal = {Journal of Applied Probability},
  volume  = {31},
  number  = {1},
  pages   = {59--75},
  year    = {1994}
}

@article{DerisaviHermannsSanders2003,
  author  = {Derisavi, S. and Hermanns, H. and Sanders, W. H.},
  title   = {Optimal state-space lumping in Markov chains},
  journal = {Information Processing Letters},
  volume  = {87},
  number  = {6},
  pages   = {309--315},
  year    = {2003},
  doi     = {10.1016/S0020-0190(03)00343-0}
}

@article{Fr1990,
  author  = {Friedman, J.},
  title   = {On the Road Coloring Problem},
  journal = {Proceedings of the American Mathematical Society},
  volume  = {110},
  number  = {4},
  pages   = {1133--1135},
  year    = {1990}
}

@book{GareyJohnson,
  author    = {Garey, M. R. and Johnson, D. S.},
  title     = {Computers and Intractability: A Guide to the Theory of NP-Completeness},
  publisher = {W. H. Freeman},
  year      = {1979}
}

@book{GrossTuckerTGT,
  author    = {Gross, J. L. and Tucker, T. W.},
  title     = {Topological Graph Theory},
  publisher = {Wiley-Interscience},
  year      = {1987}
}

@inproceedings{GuPri,
  author    = {Gusev, V. V. and Pribavkina, E. V.},
  title     = {On Synchronizing Colorings and the Eigenvectors of Digraphs},
  booktitle = {41st International Symposium on Mathematical Foundations of Computer Science (MFCS 2016)},
  series    = {Leibniz International Proceedings in Informatics (LIPIcs)},
  volume    = {58},
  pages     = {48:1--48:14},
  publisher = {Schloss Dagstuhl},
  year      = {2016}
}

@misc{GusevCIRM2015,
  author       = {Gusev, V. V.},
  title        = {On synchronizing colorings of digraphs},
  howpublished = {Talk/slides, CIRM (Centre International de Rencontres Math\'ematiques), ``Rencontres'' 1502},
  year         = {2015},
  url          = {https://www.cirm-math.fr/ProgWeebly/Renc1502/Gusev.pdf}
}

@inproceedings{GusevSzykula15,
  author    = {Gusev, V. V. and Szyku{\l}a, M.},
  title     = {On the number of synchronizing colorings of digraphs},
  booktitle = {Implementation and Application of Automata (CIAA 2015)},
  editor    = {Drewes, F.},
  series    = {Lecture Notes in Computer Science},
  volume    = {9223},
  pages     = {127--139},
  publisher = {Springer},
  year      = {2015}
}

@article{HongMiklos2023,
  author  = {Hong, L. and Mikl\'os, I.},
  title   = {A Markov chain on the solution space of edge colorings of bipartite graphs},
  journal = {Discrete Applied Mathematics},
  volume  = {332},
  pages   = {7--22},
  year    = {2023}
}

@article{Kari2003,
  author  = {Kari, J.},
  title   = {Synchronizing finite automata on Eulerian digraphs},
  journal = {Theoretical Computer Science},
  volume  = {295},
  number  = {1--3},
  pages   = {223--232},
  year    = {2003},
  doi     = {10.1016/S0304-3975(02)00405-X}
}

@book{KemenySnellFMC,
  author    = {Kemeny, J. G. and Snell, J. L.},
  title     = {Finite Markov Chains},
  publisher = {D. Van Nostrand Company},
  year      = {1960}
}

@book{LM,
  author    = {Lind, D. and Marcus, B.},
  title     = {An Introduction to Symbolic Dynamics and Coding},
  publisher = {Cambridge University Press},
  edition   = {2},
  year      = {2021}
}

@article{Marusic1981,
  author  = {Maru{\v{s}}i{\v{c}}, D.},
  title   = {On vertex symmetric digraphs},
  journal = {Discrete Mathematics},
  volume  = {36},
  number  = {1},
  pages   = {69--81},
  year    = {1981}
}

@misc{PribavkinaSandGAL2019,
  author       = {Pribavkina, E.},
  title        = {On totally synchronizing digraphs},
  howpublished = {Talk at SandGAL 2019: Semigroups and Groups, Automata, Logics, Politecnico di Milano, Cremona, Italy},
  year         = {2019}
}

@book{Seneta,
  author    = {Seneta, E.},
  title     = {Non-negative Matrices and Markov Chains},
  series    = {Springer Series in Statistics},
  publisher = {Springer},
  edition   = {2},
  year      = {2006}
}

@article{trahtman2009,
  author  = {Trahtman, A. N.},
  title   = {The road coloring problem},
  journal = {Israel Journal of Mathematics},
  volume  = {172},
  pages   = {51--60},
  year    = {2009},
  doi     = {10.1007/s11856-009-0062-5}
}

@inproceedings{volkov2008,
  author    = {M.V.~Volkov},
  title     = {Synchronizing automata and the \v{C}ern\'y conjecture},
  booktitle = {Language and Automata Theory and Applications (LATA 2008)},
  series    = {Lecture Notes in Computer Science},
  volume    = {5196},
  pages     = {11--27},
  publisher = {Springer},
  year      = {2008},
  doi       = {10.1007/978-3-540-88282-4_4}
}

@article{AraujoEtAlPLMS2016,
  author  = {J.~Ara{\'u}jo and W.~Bentz and P.J.~Cameron and G.~Royle and A.~Schaefer},
  title   = {Primitive groups, graph endomorphisms and synchronization},
  journal = {Proceedings of the London Mathematical Society},
  series  = {3},
  volume  = {113},
  number  = {6},
  pages   = {829--867},
  year    = {2016},
  doi     = {10.1112/plms/pdw040}
}

@article{AraujoCameronJCTB2014,
  author  = {J.~Ara{\'u}jo and P.J.~Cameron},
  title   = {Primitive groups synchronize non-uniform maps of extreme ranks},
  journal = {Journal of Combinatorial Theory, Series B},
  volume  = {106},
  pages   = {98--114},
  year    = {2014},
  doi     = {10.1016/j.jctb.2014.01.006}
}

@article{AraujoCameronSteinberg2017,
  author  = {J.~Ara{\'u}jo and P.J.~Cameron and B.~Steinberg},
  title   = {Between primitive and $2$-transitive: Synchronization and its friends},
  journal = {EMS Surveys in Mathematical Sciences},
  volume  = {4},
  number  = {2},
  pages   = {101--184},
  year    = {2017},
  doi     = {10.4171/EMSS/4-2-1}
}

\end{document}